\documentclass{amsart}
\usepackage{amsfonts,amsmath,amsthm,amssymb,color}
\usepackage[all]{xy}
\usepackage[pagebackref,colorlinks=true,linkcolor=blue,urlcolor=blue,citecolor=blue]{hyperref}
\usepackage{eucal,comment}
\numberwithin{equation}{section}
\providecommand{\eprint}[2][]{\href{http://arxiv.org/abs/#2}{arXiv:#2}}

\begin{document}

\newtheorem{theorem}{Theorem}[section]
\newtheorem{thm}[theorem]{Theorem}
\newtheorem{lemma}[theorem]{Lemma}
\newtheorem{proposition}[theorem]{Proposition}
\newtheorem{corollary}[theorem]{Corollary}
\newtheorem{theorem*}{Theorem}

\theoremstyle{definition}
\newtheorem{definition}[theorem]{Definition}
\newtheorem{example}[theorem]{Example}

\theoremstyle{remark}
\newtheorem{remark}[theorem]{Remark}

\newenvironment{magarray}[1]
{\renewcommand\arraystretch{#1}}
{\renewcommand\arraystretch{1}}

\newcommand{\quot}[2]{
{\lower-.2ex \hbox{$#1$}}{\kern -0.2ex /}
{\kern -0.5ex \lower.6ex\hbox{$#2$}}}

\newcommand{\mapor}[1]{\smash{\mathop{\longrightarrow}\limits^{#1}}}
\newcommand{\mapin}[1]{\smash{\mathop{\hookrightarrow}\limits^{#1}}}
\newcommand{\mapver}[1]{\Big\downarrow
\rlap{$\vcenter{\hbox{$\scriptstyle#1$}}$}}
\newcommand{\liminv}{\smash{\mathop{\lim}\limits_{\leftarrow}\,}}

\newcommand{\Set}{\mathbf{Set}}
\newcommand{\Art}{\mathbf{Art}}
\newcommand{\solose}{\Rightarrow}

\newcommand{\specif}[2]{\left\{#1\,\left|\, #2\right. \,\right\}}

\renewcommand{\bar}{\overline}
\newcommand{\de}{\partial}
\newcommand{\debar}{{\overline{\partial}}}
\newcommand{\per}{\!\cdot\!}
\newcommand{\Oh}{\mathcal{O}}
\newcommand{\sA}{\mathcal{A}}
\newcommand{\sB}{\mathcal{B}}
\newcommand{\sC}{\mathcal{C}}
\newcommand{\sD}{\mathcal{D}}
\newcommand{\sE}{\mathcal{E}}
\newcommand{\sF}{\mathcal{F}}
\newcommand{\sG}{\mathcal{G}}
\newcommand{\sH}{\mathcal{H}}
\newcommand{\sI}{\mathcal{I}}
\newcommand{\sL}{\mathcal{L}}
\newcommand{\sM}{\mathcal{M}}
\newcommand{\sP}{\mathcal{P}}
\newcommand{\sU}{\mathcal{U}}
\newcommand{\sV}{\mathcal{V}}
\newcommand{\sX}{\mathcal{X}}
\newcommand{\sY}{\mathcal{Y}}
\newcommand{\sN}{\mathcal{N}}
\newcommand{\sT}{\mathcal{T}}
\newcommand{\Aut}{\operatorname{Aut}}
\newcommand{\Id}{\operatorname{Id}}
\newcommand{\Tr}{\operatorname{Tr}}
\newcommand{\Mor}{\operatorname{Mor}}
\newcommand{\Def}{\operatorname{Def}}
\newcommand{\Fitt}{\operatorname{Fitt}}
\newcommand{\Supp}{\operatorname{Supp}}
\newcommand{\Hom}{\operatorname{Hom}}
\newcommand{\Hilb}{\operatorname{Hilb}}
\newcommand{\HOM}{\operatorname{\mathcal H}\!\!om}
\newcommand{\EXT}{\operatorname{\mathcal E}\!\!xt}
\newcommand{\DER}{\operatorname{\mathcal D}\!er}
\newcommand{\Spec}{\operatorname{Spec}}
\newcommand{\Der}{\operatorname{Der}}
\newcommand{\Tor}{{\operatorname{Tor}}}
\newcommand{\Ext}{{\operatorname{Ext}}}
\newcommand{\End}{{\operatorname{End}}}
\newcommand{\END}{\operatorname{\mathcal E}\!\!nd}
\newcommand{\Image}{\operatorname{Im}}
\newcommand{\coker}{\operatorname{coker}}
\newcommand{\tot}{\operatorname{tot}}
\newcommand{\ten}{\bigotimes}
\newcommand{\mA}{\mathfrak{m}_{A}}

\newcommand{\somdir}[2]{\hbox{$\mathrel
{\smash{\mathop{\mathop \bigoplus\limits_{#1}}
\limits^{#2}}}$}}
\newcommand{\tensor}[2]{\hbox{$\mathrel
{\smash{\mathop{\mathop \bigotimes\limits_{#1}}
^{#2}}}$}}
\newcommand{\symm}[2]{\hbox{$\mathrel
{\smash{\mathop{\mathop \bigodot\limits_{#1}}
^{#2}}}$}}
\newcommand{\external}[2]{\hbox{$\mathrel
{\smash{\mathop{\mathop \bigwedge\limits_{#1}}
^{\!#2}}}$}}

\renewcommand{\Hat}[1]{\widehat{#1}}
\newcommand{\dual}{^{\vee}}
\newcommand{\desude}[2]{\dfrac{\de #1}{\de #2}}

\newcommand{\A}{\mathbb{A}}
\newcommand{\N}{\mathbb{N}}
\newcommand{\R}{\mathbb{R}}
\newcommand{\Z}{\mathbb{Z}}
\renewcommand{\H}{\mathbb{H}}
\renewcommand{\L}{\mathbb{L}}
\newcommand{\proj}{\mathbb{P}}
\newcommand{\K}{\mathbb{K}\,}
\newcommand\C{\mathbb{C}}
\newcommand\Del{\operatorname{Del}}
\newcommand\Tot{\operatorname{Tot}}
\newcommand\Grpd{\mbox{\bf Grpd}}

\newcommand{\g}{\mathfrak{g}}

\newcommand\é{\'e}
\newcommand\è{\`e}
\newcommand\à{\`a}
\newcommand\ì{\`i}
\newcommand\ù{\`u}
\newcommand\ò{\`o }


\newcommand{\rh}{\rightarrow}
\newcommand{\contr}{{\mspace{1mu}\lrcorner\mspace{1.5mu}}}

\newcommand{\bi}{\boldsymbol{i}}
\newcommand{\bl}{\boldsymbol{l}}

\newcommand{\MC}{\operatorname{MC}}
\newcommand{\TW}{\operatorname{TW}}

\newenvironment{acknowledgement}{\par\addvspace{17pt}\small\rm
\trivlist\item[\hskip\labelsep{\it Acknowledgement.}]}
{\endtrivlist\addvspace{6pt}}


\title{On deformations of pairs (manifold, coherent sheaf)}
\date{\today}

\author{Donatella Iacono}
\address{\newline  Universit\`a degli Studi di Bari,
\newline Dipartimento di Matematica,
\hfill\newline Via E. Orabona 4,
I-70125 Bari, Italy.}
\email{donatella.iacono@uniba.it}
\urladdr{\href{http://www.dm.uniba.it/~iacono/}{www.dm.uniba.it/~iacono/}}

\author{Marco Manetti}
\address{\newline
Universit\`a degli studi di Roma ``La Sapienza'',\hfill\newline
Dipartimento di Matematica \lq\lq Guido
Castelnuovo\rq\rq,\hfill\newline
P.le Aldo Moro 5,
I-00185 Roma, Italy.}
\email{manetti@mat.uniroma1.it}
\urladdr{\href{http://www.mat.uniroma1.it/people/manetti/}{www.mat.uniroma1.it/people/manetti/}}

\begin{abstract}
We analyse infinitesimal deformations of pairs $(X,\sF)$ with  $\sF$  a coherent sheaf on a smooth projective variety  $X$ over an algebraically closed field of characteristic $0$. We describe a differential graded Lie algebra controlling the deformation problem, and we prove an analog of a Mukai-Artamkin Theorem about the trace map.

\end{abstract}


\maketitle

\section{Introduction}

Let $\sF$ be a coherent sheaf on a smooth projective variety $X$ over an algebraically closed field of characteristic $0$; in
particular  $\sF$ admits a finite locally free resolution.
Then, it is defined the sheaf trace morphism  
\begin{equation}\label{eq trace map}
\Tr\colon  \HOM_{\Oh_X} (\sF,\sF) \to \Oh_X\,,
\end{equation}
and the trace maps
\begin{equation}\label{eq trace map ext}
\Tr^i\colon  \Ext_X^i (\sF,\sF) \to H^i(X,\Oh_X),
\end{equation}
which are the morphisms  induced in hypercohomology by  $\Tr$ see e.g. \cite{Arta,Kobabook,casson}.

  The maps $\Tr^1$ and $\Tr^2$ have a clear interpretation in the setting of deformation theory of $\sF$ and of the determinant bundle $\det \sF$ of $\sF$. Every deformation of $\sF$ induces naturally a deformation of the determinant line bundle
$\det \sF$; the vector spaces $\Ext_X^i (\sF,\sF)$, $i=1,2$, are the tangent and obstruction spaces 
of the functor $\Def_{\sF}$ of deformations of $\sF$; the spaces $H^i(X,\Oh_X)$, $i=1,2$, 
are  the tangent and obstruction spaces 
of the functor $\Def_{\det \sF}$ of deformations of $\det \sF $. Then, the maps $\Tr^i$, $i=1,2$, are induced by the 
natural transformation  $\Def_{\sF}\to \Def_{\det \sF}$. 

Moreover, since we are in characteristic $0$, the Picard functor is unobstructed and  this allows to prove that $\Tr^2$ annihilates all
the obstructions  to deformations of $\sF$. As a consequence, we have the following well known result.

\begin{theorem}[Mukai-Artamkin \cite{Arta,mukai84}]\label{thm.artamkin}
In the above situation, if $\Tr^2$ is injective, then the deformation functor $\Def_{\sF}$ is unobstructed. If
$\Tr^1$ is surjective and $\Tr^2$ is injective, then the natural transformation 
$\Def_{\sF}\to \Def_{\det \sF}$ is smooth.  
\end{theorem}

In their papers, Mukai and Artamkin assume that $\sF$ is simple, however it is immediate to see that this assumption is used only 
for the pro-representability of the functor $\Def_{\sF}$, and hence for the existence of the universal deformation of $\sF$, while the proof of the smoothness of the semiuniversal deformation works also without the simpleness assumption.

One of the main goal of this paper is to extend the previous result to the case of infinitesimal deformations of pairs  $(X,\sF)$, consisting of deformations of both the variety $X$ and the coherent sheaf $\sF$.

As a first result, we describe a well defined homotopy class of differential graded Lie algebras associated with the deformations of the pair $(X,\sF)$.

For every quasi-coherent sheaf of DG-Lie algebras $\sL^*$ on an algebraic variety $X$ it is defined, up to homotopy equivalence, the DG-Lie algebras of derived sections $R\Gamma(X,\sL^*)$  and we shall say that a deformation problem is controlled by $\sL^*$ if it is controlled in the usual way by the DG-Lie algebra $R\Gamma(X,\sL^*)$. A canonical representative for $R\Gamma(X,\sL^*)$ is given 
by the totalization, in the simplicially enriched model category of DG-Lie algebras, of the 
cosimplicial space of \v{C}ech cochains of $\sL^*$, with respect to any open affine cover of $X$, see 
Section~\ref{sec.review}.

For a  coherent sheaf $\sF$ on a projective smooth variety $X$, it is known 
that the 
deformations of  $\sF$ are controlled by the sheaf of DG-Lie algebra of 
endomorphisms of any finite locally free resolution of $\sF$  \cite{FIM}. 

It is also known that if $\sE$ is a locally free sheaf on a smooth variety $X$, then the 
deformations of the pair $(X,\sE)$ are controlled by the sheaf of Lie algebra of 
first order differential operators on $\sE$ with principal symbol \cite{huang,manetticinese,Sernesi}. For a general coherent sheaf $\sF$, the sheaf of differential operators on $\sF$ is equally well defined 
\cite[Section 16.8]{egaIV4} but, in general,  it is not the right object controlling the deformations of pairs. 

Therefore, we first introduce the right algebraic object associated with the deformation problem: we shall call this object the module of derivations of pairs. It extends the sheaf of differential operators on $\sF$, involving a  finite   locally free resolution of $\sF$. This allows us to define a  coherent sheaf of DG-Lie algebras over $X$ controlling the deformations of the pair $(X,\sF)$ (Theorem \ref{thereom D1 controlla def coppia X F}).

 In particular, the cohomology groups $T^i_{(X,\sF)}$  of the associated  DG-Lie algebra fits into a long exact sequence:
\[ \cdots\to \Ext^i_X(\sF,\sF)\to T^i_{(X,\sF)}\to H^i(X,\Theta_X)\to \Ext^{i+1}_X(\sF,\sF)\to \cdots\;,\]
where $\Theta_X$ denotes the tangent sheaf of $X$.
In particular, we recover the well known fact that if $\Ext^2_X(\sF,\sF)=0$, then the natural transformation 
$\Def_{(X,\sF)}\to  \Def_X$ is smooth, since it is surjective on tangent spaces ($T^1_{(X,\sF)}\to H^1(X,\Theta_X)$) and injective on obstruction spaces ($T^2_{(X,\sF)}\to H^2(X,\Theta_X)$).

Then, we devote our attention to  the natural transformations 
\[\Def_{(X,\sF)}\to \Def_{(X,\det \sF)}\to \Def_X.\] 
In particular, we describe an extension of the trace maps to the DG-Lie algebra of differential operators with principal symbol (Theorem \ref{thm.tracediagram}).

Finally,  we are able to prove the following result (Theorem \ref{teo finale su traccia}), which is the analog of Theorem~\ref{thm.artamkin}, for deformations of pairs.

\begin{theorem}\label{thm.maintheorem}
Let $ \sF$ be a coherent sheaf on  a projective smooth variety $X$ defined over an algebraically closed field of characteristic 0. Consider the trace maps $\Tr^i\colon  \Ext_X^i (\sF,\sF) \to H^i(X,\Oh_X)$: 
\begin{enumerate}
\item if the map  
$\Tr^1$ is surjective and the map $\Tr^2$ is injective, then the natural transformation 
$\Def_{(X,\sF)}\to \Def_{(X,\det \sF)}$ is smooth;

\item if $\Tr^2$ is injective and $\Def_{(X,\det \sF)}$ is unobstructed, then $\Def_{(X,\sF)}$ is unobstructed;

\item if $H^0(X,\Theta_X)=0$ and 
$\Ext_X^0(\sF,\sF)=\K$ (e.g., if $\sF$ is simple), then $\Def_{(X,\sF)}$ is pro-representable.
\end{enumerate}
\end{theorem}

For instance if $\sF$ is a simple coherent sheaf \cite[p. 101]{mukai84} of positive rank over a surface $S$ with trivial canonical bundle, then 
$\Tr^1$ is surjective, $\Tr^2$ is injective and therefore
the natural transformation $\Def_{(X,\sF)}\to \Def_{(X,\det \sF)}$ is smooth.  
When $\sF$ is locally free, Theorem~\ref{thm.maintheorem} was already proved in \cite{huang} by using 
transcendental methods, and hence over the field of complex numbers.  

Our proof of Theorem~\ref{thm.maintheorem} is almost entirely algebraic and it relies on the  the explicit description of a   DG-Lie algebra controlling the deformations of the pair $(X,\sF)$ and the extension of the trace maps.

Another derived extension of the determinant map was introduced in \cite{timo} from the derived stack of perfect complex to derived stack of line bundle.

The paper goes as follows. The first sections of this paper are devoted to the introduction and the  study of the first properties of  the module of derivations of pairs; we prove that it behaves very well with respect to all the canonical constructions, injective and projective resolutions, de Rham complexes and Fitting stratifications.
In Section \ref{sezione definizione traccia} , we define a  coherent sheaf of DG-Lie algebras associated with derivations of pairs and the extension of the   trace map  (Theorem \ref{thm.tracediagram}), that generalises the classical one given in \eqref{eq trace map}.
 Section \ref{sezione DGLA e tot} is included for readers's convenience: here we review the relevant notions on deformation functors associated with a differential graded Lie algebras and with a semicosimplicial DG-Lie algebras. 
In Section \ref{sezione deformazioni coppie geoemtrica}, we analyse the infinitesimal deformations of pairs, describing the 
DG-Lie algebra that controls these deformations (Theorem \ref{thereom D1 controlla def coppia X F}) and we prove the main theorem  (Theorem \ref{thm.maintheorem}) about the trace map (Theorem \ref{teo finale su traccia}).

\bigskip
\section{Derivations and automorphisms of pairs}
\label{sec.derivationpairs}

Let $A\to R$ be a morphism of unitary commutative rings and $M,N$ two $R$-modules. We shall denote by 
$\Hom_A(M,N)$ (resp.: $\Hom_R(M,N)$) the $R$-module of $A$-linear (resp.: $R$-linear) maps $M\to N$: 
the $R$-module structure on $\Hom_A(M,N)$
is induced by the $R$-module structure on $N$. We shall denote by $\Der_A(R,N)=\Hom_R(\Omega_{R/A},N)$ the $R$-module of 
$A$-linear derivations $R\to N$. We shall refer to the $R$-module $R\oplus M$ as the trivial extension whenever $R\oplus M$ is considered as a commutative ring  equipped  with the product $(r,m)(s,n)=(rs,rn+sm)$, see e.g., \cite[p.10]{Sernesi}.

\begin{definition}\label{def.derivationspairs} 
Let $A\to R$ be a morphism of unitary commutative rings and $M$ an $R$-module. The $R$-module of 
$A$-derivations of the pair $(R,M)$ is defined as
\[ D_A(R,M)=\left\{(h,u)\in \Der_A(R,R)\times \Hom_A(M,M)\;\middle|\;
\begin{array}{l}
 u(rm)-ru(m)=h(r)m,\smallskip\\
 \text{for every } r\in R,\; m\in M\end{array}\right\}.\]
\end{definition}

\begin{remark} In the setup of the above definition:
\begin{enumerate}

\item if $r_1,r_2,\ldots\in R$ generate $R$ as an $A$-algebra and $m_1,m_2,\ldots\in M$ generate $M$ as $R$-module, then every $(h,u)\in  D_A(R,M)$ is uniquely determined by 
$h(r_1),h(r_2),\ldots$ and $u(m_1),u(m_2),\ldots$;

\item if  $R\oplus M$ is the trivial extension of $R$ by $M$, then a pair $(h,u)\in \Hom_A(R,R)\times \Hom_A(M,M)$ belongs to $D_A(R,M)$ if and only if 
the map $R\oplus M\to R\oplus M$, $(r,m)\mapsto (h(r),u(m))$ is an $A$-derivation.
\end{enumerate}
\end{remark}

\begin{lemma}\label{lem.esattapercoerenza} 
There exists an exact sequence of $R$-modules
\[ 0\to D_A(R,M)\to \Der_A(R\oplus M,R\oplus M)\xrightarrow{\Phi} \Hom_{R}(M,R)\oplus \Der_A(R,M)\,,\]
where the $R$-module structure on the derivations of the trivial extension $R\oplus M$ is induced by the 
inclusion $R\to R\oplus M$. 
\end{lemma}

\begin{proof} 
Every element of 
$\Hom_A(R\oplus M,R\oplus M)$ is represented by a matrix of $A$-linear maps
$\begin{pmatrix}a&b\\ c&d\end{pmatrix}$, with 
\[ a\in \Hom_A(R,R),\quad b\in \Hom_A(M,R),\quad c\in \Hom_A(R,M),\quad d\in \Hom_A(M,M),\]
and then there exists a natural isomorphism of $R$-modules
\[ \Hom_A(R\oplus M,R\oplus M)=\Hom_A(R,R)\oplus\Hom_A(M,R)\oplus\Hom_A(R,M)\oplus\Hom_A(M,M)\,.\]
We have already noticed that $(h,u)\in D_A(R,M)$ if and only if  
$\begin{pmatrix}h&0\\ 0&u\end{pmatrix}\in \Der_A(R\oplus M,R\oplus M)$.
Therefore, defining $\Phi\begin{pmatrix}a&b\\ c&d\end{pmatrix}=(b,c)$, it is straightforward to see that 
if $\begin{pmatrix}a&b\\ c&d\end{pmatrix}$ is a derivation, then $b$ is $R$-linear and $c$ is a derivation.
It also easy to prove that the image of $\Phi$ is  
\[ \{b\in \Hom_{R}(M,R)\mid b(m)n+mb(n)=0\quad \forall\; m,n\in M\}\oplus \Der_A(R,M)\,.\]
\end{proof}

\begin{lemma}\label{lem.basicexactsequence} 
In the setup of Definition~\ref{def.derivationspairs}, let us denote by 
\[q\colon D_A(R,M)\to \Hom_{A}(M,M)\qquad \mbox{ and }\qquad \alpha\colon D_A(R,M)\to\Der_A(R,R),\]
the  projection maps restricted to  $D_A(R,M)$:
\begin{enumerate}

\item if $M$ is a faithful module, i.e., if $ann(M)=0$, then $q$ 
is injective and its image is the submodule of 
differential operators of first order with principal symbol;

\item there exists an  exact sequence of $R$-modules
\begin{equation}\label{equ.basicexactsequence} 0\to \Hom_R(M,M)\to  D_A(R,M)\xrightarrow{\,\alpha\,}\Der_A(R,R)\;.\end{equation}

\end{enumerate}

\end{lemma}

\begin{proof} The only nontrivial statement is the one concerning the image of $q$. 
Recall that a morphism $v\in \Hom_A(M,M)$ is a differential operator of first order if for every  $r\in R$  
the map 
\[[v,r]\colon M\to M,\quad [v,r](m)=v(rm)-rv(m),\]
is a morphism of $R$-modules. In this case the  symbol $\sigma(v)$ is defined as the map
\[ \sigma(v)\colon R\to \Hom_R(M,M),\qquad  \sigma(v)(r)=[v,r]\,,\] 
and it is called principal if $\sigma(v)(r)$ is a scalar multiple of the identity for every $r\in R$. 
Notice that the symbol $\sigma(v)$ is an $A$-derivation, since for every $r,s\in R$ we have
\[ \sigma(v)(rs)=[v,rs]=[v,r]s+r[v,s]\,.\]
If $(h,u)\in D_A(R,M)$, then for every $r\in R$ 
\[ [u,r]=h(r)\Id_M\in  \Hom_R(M,M)\,.\]
Conversely, let
$v\in \Hom_A(M,M)$ be a first order differential operator such that $\sigma(v)(R)\subseteq R\Id_M$. Since $M$ is faithful,  for every $r\in R$ there exists a unique
$h(r)\in R$ such that $\sigma(v)(r)=[v,r]=h(r)\Id_M$. 
Since $\sigma(v)$ is an $A$-derivation, also $h\colon R\to R$ is an $A$-derivation and  therefore 
$(h,v)\in D_A(R,M)$.
\end{proof}

The restriction $\alpha\colon D_A(R,M)\to \Der_A(R,R)$ of the projection on the first factor is called the
anchor map of the pair.  
It is plain that $D_A(R,M)$ is a Lie subalgebra of $\Der_A(R,R)\times \Hom_A(M,M)$ and we have the relation 
\begin{equation} 
[x,ry]=\alpha(x)(r)y+r[x,y],\qquad x,y\in D_A(R,M),\quad r\in R\,,\end{equation}
called Poisson identity. Keep attention to the fact that the bracket on $\Der_A(R,R)$, and then on  
$D_A(R,M)$, is bilinear over $A$ and not over $R$.

\begin{example}[Lie derivative]\label{ex.Liederivative} 
Let $A\to R$ be a morphism of unitary commutative rings. Then, every 
$h\in \Der_A(R,R)$ gives a canonical element $(h,L_h)\in D_A(R,\Omega_{R/A})$ uniquely determined by 
the equation $L_{h}(dx)=d(h(x))$, $x\in R$.  

Recall that we may define the module of K\"{a}hler differentials as 
$\Omega_{R/A}=I/I^2$, where $I$ is the kernel of the multiplication map $R\otimes_A R\to R$, the differential 
$dx$ is the class of $x\otimes 1-1\otimes x$ in $I/I^2$ and the $R$-module structure is induced by the morphism of $A$-algebras $R\to R\otimes_A R$, $r\mapsto r\otimes 1$,  see e.g., \cite[Section 25]{Matsubook}. 
For every $h\in \Der_A(R,R)$  we define $L_{h}\colon I/I^2\to I/I^2$ as the factorization to the quotient of 
the derivation 
\[k\colon R\otimes_A R\to R\otimes_A R,\qquad k(x\otimes y)=h(x)\otimes y+x\otimes h(y)\,.\]
The equation $L_{h}(dx)=d(h(x))$ is trivially satisfied. For every $r,s,x\in R$, we have 
\[\begin{split} 
L_h(rsdx)&=k((r\otimes 1)(sx\otimes 1-s\otimes x))\\
&=(h(r)\otimes 1)(sx\otimes 1-s\otimes x)+(r\otimes 1)k(sx\otimes 1-s\otimes x)=h(r)sdx+rL_{h}(sdx)\,,
\end{split}\]
and then $(h,L_h)\in D_A(R,\Omega_{R/A})$.
\end{example}

The definition of $D_A(R,M)$ extends naturally to $D_A(R,M_{\bullet})$, where $M_\bullet$ is a diagram of $R$-modules over a small category $I$. Here we are mainly interested in two  cases. The former is when 
$I$ is just a set, a diagram of $R$-modules is just a collection $\{M_i\}$, $i\in I$, and 
$D_A(R,M_{\bullet})$ is the limit of the diagram of anchor maps $\alpha\colon D_A(R,M_i)\to \Der_A(R,R)$. For later use, it is notational convenient to denote 
\[\begin{split} {\prod_{i}}^{\times}D_A(R,M_i)&=D_A(R,M_{\bullet})=
\lim_{i}(\alpha\colon D_A(R,M_i)\to \Der_A(R,R))\\
&=\{(h,\ldots,u_i,\ldots)\mid (h,u_i)\in D_A(R,M_i),\quad \forall\; i\in I\}\,.\end{split}\]
     
The latter is when
$M_\bullet=\{M_1\xrightarrow{f}M_2\}$ is a morphism of two $R$-modules, then $D_A(R,M_{\bullet})$ is the set of triples $(h,u_1,u_2)$, with $(h,u_1)\in D_A(R,M_{1})$,
$(h,u_2)\in D_A(R,M_{2})$, and $fu_1=u_2f$.  Therefore, there exists an exact sequence of $R$-modules
\[  0\to D_A(R,M_{\bullet})\to D_A(R,M_1)\times_{\Der_A(R,R)}D_A(R,M_2)\xrightarrow{\;\Phi\;}\Hom_R(M_1,M_2),\]
where
\[ \Phi((h,u_1),(h,u_2))=fu_1-u_2f\,.\]

\begin{lemma}\label{lem.tensor_e_hom} 
Let $A\to R$ be a morphism of unitary commutative rings and $M,N$ two $R$-modules. 
Then, we have two natural $R$-linear morphisms of Lie algebras:
\begin{equation}\label{equ.tensorderiv} 
D_A(R, M\otimes_RN)\xleftarrow{\Phi} D_A(R,M)\times_{\Der_A(R,R)}D_A(R,N)\xrightarrow{\Psi} D_A(R, \Hom_R(M,N)),\end{equation}
\[ \begin{split}\Phi((h,u),(h,v))&=(h,u\otimes_R\Id_N+\Id_M\otimes_Rv),\\
\Psi((h,u),(h,v))&=(h,f\mapsto vf-fu)\,.\end{split}\]

\end{lemma}

\begin{proof} Straightforward. Notice that $u\otimes_R\Id_N+\Id_M\otimes_Rv$ is well defined, although 
$u,v$ are not $R$-linear and the two addends $u\otimes_R\Id_N$, $\Id_M\otimes_Rv$ are not defined. 
Similarly $vf-fu$ is a morphism of $R$-modules, although $vf$ and $fu$ are only $A$-linear. 
\end{proof}

As a particular case of Lemma~\ref{lem.tensor_e_hom} we obtain a natural $R$-linear morphisms of Lie algebras (the transpose)
\begin{equation}\label{equ.transposition}
\begin{split}
(-)^T\colon D_A(R,M)&\to D_A(R, \Hom_R(M,R)),\\
(h,u)&\mapsto (h,u)^T=\Psi((h,u),(h,h))=(h,f\mapsto hf-fu)\,.\end{split}
\end{equation}

\begin{definition}\label{def.leibnizextension}
The Leibniz extension of a derivation of pair  $(h,u)\in D_A(R,M)$ is the sequence $(h,u_n)\in 
D_A(R,\wedge_R^nM)$, $n\ge 0$,  uniquely determined by the formulas:
\begin{equation}\label{equ.leibnizextension} 
u_0=h,\qquad  u_n(m_1\wedge\cdots\wedge m_n)=\sum_{i=1}^nm_1\wedge\cdots\wedge u(m_i)\wedge\cdots\wedge m_n
\,.\end{equation}
\end{definition} 

The Leibniz extension is properly defined: by the universal property of wedge products, the Formula~\eqref{equ.leibnizextension} defines a sequence of $A$-linear maps    
\[ \tilde{u}_n\colon  \wedge_A^nM\to \wedge_A^nM\,.\]
Now $\wedge_R^nM$ is the quotient of $\wedge_A^nM$ by the $A$-submodule $H$ generated by all the elements \[m_1\wedge \cdots\wedge rm_i\wedge m_{i+1}\wedge\cdots\wedge m_n-
m_1\wedge \cdots\wedge m_i\wedge rm_{i+1}\wedge\cdots\wedge m_n,\quad  r\in R,\; m_j\in M,\; 0<i<n\,,\]
and it is immediate to verify that $\tilde{u}_n(H)\subset H$ and then that $(h,\tilde{u}_n)$ factors to a derivation of pair  $(h,u_n)\in 
D_A(R,\wedge_R^nM)$. The Leibniz extension is functorial in the following sense: 
given a morphism of $R$-modules $f\colon M\to N$ and $(h,u)\in D_A(R,M)$, $(h,v)\in D_A(R,N)$ such that 
$vf=fu$, then  $v_nf^{\wedge n}=f^{\wedge n}u_n$ for every $n$. Moreover, for every $n$, the map 
\[ D_A(R,M)\to D_A(R,\wedge^n_RM),\qquad (h,u)\mapsto (h,u_n),\]
is a morphism of Lie algebras. This follows immediately from the fact that, 
for every $(h,u),(k,v)\in D_A(R,M)$ we have
\[\begin{split}
u_nv_n(m_1\wedge\cdots\wedge m_n)&=\sum_{i=1}^nm_1\wedge\cdots\wedge uv(m_i)\wedge\cdots\wedge m_n\\
&\qquad+\sum_{i<j}m_1\wedge\cdots u(m_i)\wedge\cdots \wedge v(m_j)\cdots\wedge m_n\\
&\qquad+\sum_{i<j}m_1\wedge\cdots v(m_i)\wedge\cdots \wedge u(m_j)\cdots\wedge m_n\;. \end{split}\]

\begin{definition}\label{def.tracemodule} 
If $M$ is a free module of rank $n$ and $(h,u)\in D_A(R,M)$, we shall call $(h,u_n)\in D_A(R,\wedge_R^nM)$ the 
trace of $(h,u)$.
\end{definition} 

The name trace is motivated by the fact that if $h=0$, i.e., if $u\in \Hom_R(M,M)$, then 
$u_n$ is the multiplication by the trace of $u$.

\medskip
\subsection*{Automorphisms of pairs.}
Let $A\to R$ be a morphism of commutative unitary rings and let $M$ be an $R$-module.
We shall denote by $\Aut_A(R)$ the group of $A$-linear automorphisms of $R$, i.e., the automorphism group of the $A$-algebra $R$, and by $\Aut_A(M)$ (resp.: $\Aut_R(M)$) the group of $A$-linear (resp.: $R$-linear) automorphisms of $M$.

\begin{definition}\label{def.automorphismpairs} 
The group  of 
$A$-linear automorphisms of the pair $(R,M)$ is defined as the subgroup 
$\Aut_A(R,M)\subset\Aut_A(R)\times \Aut_A(M)$  
of pairs $(\theta,\phi)$  such that
$\phi(rm)=\theta(r)\phi(m)$ for every $r\in R$, $m\in M$.
\end{definition}

Let $R\oplus M$ be the trivial extension of $R$ by $M$. As in the proof of Lemma~\ref{lem.esattapercoerenza}, there exists a natural inclusion 
\[ \Hom_A(R,R)\times \Hom_A(M,M)\subset \Hom_A(R\oplus M,R\oplus M)\]
and it is immediate to 
see that an element of  $\Hom_A(R,R)\times \Hom_A(M,M)$ is an automorphism 
of the pair $(R,M)$ if and only if it is an automorphism of the $A$-algebra $R\oplus M$.

The analog of the anchor map is the group homomorphism 
\[\alpha\colon \Aut_A(R,M)\to \Aut_A(R),\qquad \alpha(\theta,\phi)=\theta,\]
whose kernel is $\Aut_R(M)$.

There exists the analog of Lemma~\ref{lem.tensor_e_hom} and of the Leibniz extension for automorphisms of pairs.
It is straightforward to verify that there exist two natural group homomorphisms

\begin{equation}\label{equ.tensorauto} 
\Aut_A(R, M\otimes_RN)\xleftarrow{\Phi} \Aut_A(R,M)\times_{\Aut_A(R)}\Aut_A(R,N)\xrightarrow{\Psi} \Aut_A(R, \Hom_R(M,N)),\end{equation}
\[ \begin{split}\Phi((\theta,\phi),(\theta,\psi))&=(\theta,\phi\otimes\psi),\\
\Psi((\theta,\phi),(\theta,\psi))&=(\theta,f\mapsto \psi f\phi^{-1})\,.\end{split}\]

Similarly, every  $(\theta,\phi)\in \Aut_A(R,M)$ gives a sequence $(\theta,\phi_n)\in 
\Aut_A(R,\wedge_R^nM)$, $n\ge 0$,  uniquely determined by the formulas:
\begin{equation} 
\phi_0=\theta,\qquad  \phi_n(m_1\wedge\cdots\wedge m_n)=\phi(m_1)\wedge\cdots\wedge  \phi(m_n)
\,.\end{equation}
If $M$ is free of rank $n$, we write $(\theta,\phi_n)=\det(\theta,\phi)$: when $\theta=\Id$ we recover the usual notion of the determinant of a $R$-linear endomorphism.

Assume now that $A$ contains the field $\mathbb{Q}$ of rational numbers and let 
$(h,u)\in D_A(R,M)\subset \Der_A(R\oplus M,R\oplus M)$ be a nilpotent derivation of pairs.
Then, also its exponential 
\[ \exp(h,u)=\left(\sum_{n=0}^\infty \frac{h^n}{n!},\sum_{n=0}^\infty \frac{u^n}{n!}\right)\]
preserves the direct sum decomposition $R\oplus M$ and then 
\[ \exp(h,u)\in \Aut_A(R,M)\subset \Aut_A(R\oplus M)\,.\]
It is plain that the exponential commutes with the anchor maps and the usual properties of the exponential imply that $\exp$ commutes with the morphisms $\eqref{equ.tensorderiv}$ and $\eqref{equ.tensorauto}$. 
The exponential also commutes with Leibniz extensions of derivations and automorphisms: this is immediate from the previous remark  since the assumption $\mathbb{Q}\subset A$ implies that $\bigwedge^n_RM$ is a direct summand of 
$\bigotimes^n_RM$. In particular, the exponential of the trace is the determinant of the exponential.

\begin{lemma}\label{lemma exp auto}
Let $R$ be a commutative unitary algebra over a field $\K$ of characteristic 0, and let 
$M$ be an $R$-module. Then, for every local Artin $\K$-algebra $A$ with residue field $\K$,  
the group $\exp(D_{\K}(R,M)\otimes\mathfrak{m}_A)$ is naturally isomorphic to the group 
of $A$-linear automorphisms of the pair $(R\otimes A,M\otimes A)$ lifting the identity on $(R,M)$.
\end{lemma}

\begin{proof} We have already noticed that via the diagonal inclusion 
\[ \Hom_\K(R,R)\times \Hom_\K(M,M)\subset \Hom_\K(R\oplus M,R\oplus M),\]
a couple $(h,u)\in \Hom_\K(R,R)\times \Hom_\K(M,M)$ is a $\K$-linear derivation (resp.: automorphism) 
of the pair $(R,M)$ if and only if it is a $\K$-linear derivation (resp.: automorphism) 
of the trivial extension $R\oplus M$. Thus, the lemma is an immediate consequence of the well known fact (see e.g. \cite[Proposition 5.44]{ManRendiconti}) that for every commutative unitary $\K$-algebra $S$, the group 
$\exp(\Der_{\K}(S,S)\otimes\mathfrak{m}_A)$ is naturally isomorphic to the group 
of $A$-linear automorphisms of $S\otimes A$ lifting the identity on $S$.
\end{proof}

\bigskip
\section{Lifting to resolutions} 
\label{sec.lifting}

Throughout this section, $A\to R$ is a fixed morphism of unitary commutative rings. For every 
$(h,u)\in D_A(R,M)$ the derivation $h$ preserves the annihilator of $M$: in fact, if $r\in R$ and 
$rm=0$ for every $m\in M$, then also 
$h(r)m=u(rm)-r u(m)=0$, for every $m\in M$. 

The above consideration shows that in general the anchor map $\alpha\colon D_A(R,M)\to \Der_A(R,R)$ is not surjective. 
For instance, if $M=R/I$, with $I$ an ideal of $R$, then the image of $\alpha$ is the submodule of derivations  preserving the ideal $I$: if $(h,u)\in D_A(R,M)$ we have seen that $h$ preserves the annihilator of $M$, which is precisely the ideal $I$.
Conversely, any derivation $h\in \Der_A(R,R)$ has a canonical lifting  $(h,h) \in D_A(R,R)$, and   
if  $h(I) \subset I$, then  $(h,h)$ factors to an element of $D_A(R,M)$.

\begin{lemma}\label{lem.liftingproperty} 
In the above setup,  if 
$f\colon P\to M$ is a surjective morphism of $R$-modules with $P$ projective, then 
every $(h,u)\in D_A(R,M)$ lifts to an element $(h,v)\in D_A(R,P)$ such that $fv=uf$.
In particular: 
\begin{enumerate}

\item every derivation of pairs lifts to any projective resolution; 

\item if $P$ is a projective  $R$-module, then the anchor map 
$\alpha\colon  D_A(R,P)\to \Der_A(R,R)$ is surjective.
\end{enumerate}
\end{lemma}

\begin{proof} Consider first the case when $P$ is a free $R$-module with basis
$\{e_i\}$. Choosing  elements $v_i\in P$ such that 
$f(v_i)=u(f(e_i))\in M$ for every $i$, then 
the $A$-linear map  
\[ v\colon P\to P,\qquad v(\sum_{i}a_ie_i)=\sum_{i}a_iv_i+h(a_i)e_i,\]
has the required properties. If $P$ is not free, since every projective module is a direct summand of a free module, there exist a free module $F$ together with two morphisms $i\colon P\to F$, $g\colon F\to P$ such that $gi=\Id_P$. Since $fg\colon F\to M$ is surjective
there exists $(h,w)\in D_A(R,F)$ lifting $(h,u)$ and it is sufficient to take 
$(h,v)=(h,gwi)\in D_A(R,P)$. If $M=0$ then for every $h\in \Der_A(R,R)$ we have 
$(h,0)\in D_A(R,0)$ and the above computation gives the surjectivity of the anchor map
$\alpha\colon  D_A(R,P)\to \Der_A(R,R)$.
\end{proof}

\begin{lemma}\label{lem.extensionproperty} 
In the above setup,  if 
$g\colon M\to J$ is an injective morphism of $R$-modules with $J$ injective, then 
every $(h,u)\in D_A(R,M)$ extends  to an element $(h,v)\in D_A(R,J)$ such that $vg=gu$.
In particular: 
\begin{enumerate}

\item every derivation of pairs extends to any injective resolution; 

\item if $J$ is a injective  $R$-module, then the anchor map 
$\alpha\colon  D_A(R,J)\to \Der_A(R,R)$ is surjective.
\end{enumerate}
\end{lemma}

\begin{proof} 
Let $q\colon J \to I$ be an injective morphism with   $I$ injective as an $A$-module, then the composition $q g\colon M\to J\to I$ is also injective.  
Consider the $R$-module $\Hom_A(R, I)$, where the $R$-module structure is given by $(t\psi)(r)=\psi(tr)$, for any $t,r \in R$ and any $\psi \in \Hom_A(R, I)$. Then, the map 
\[\beta\colon J\to \Hom_A(R, I) \qquad  j\mapsto \beta(j)(r)= q(rj) \in I, \qquad \forall \,r \in R\,,\]
is an injective $R$-linear morphism.
Since $J$ is injective as $R$-module, there exists a splitting $\gamma\colon\Hom_A(R, I) \to J$, such that $\gamma \beta=\Id_J$.
Let  $(h,u)\in D_A(R,M)$ be a fixed derivation of the pair; since $I$ is injective as $A$-module and $u$ is a morphism of $A$-modules, there exists a morphism  $w\in \Hom_A(I,I)$ such that $wqg=qgu$:
\[ \begin{matrix}\xymatrix{M\ar[d]_u\ar[r]^{qg}&I\ar[d]_w\\
M\ar[r]^{qg}&I}\end{matrix}\;.\] 
Consider now the map $w_h\colon \Hom_A(R,I)\to \Hom_A(R,I)$ defined by 
$w_h(\psi)=w\circ \psi-\psi \circ h$, for any $\psi \in \Hom_A(R,I)$. Then, $(h,w_h) \in  D_A(R,\Hom_A(R,I) )$ since 
$w_h$ is $A$-linear and   $w_h(t\psi)=t w_h(\psi) +h(t)\psi$, for any $t,r\in R$; indeed,  we have:
\[w_h(t\psi)(r)= (w\circ t\psi-t\psi \circ h)(r)= w (\psi(tr))-\psi(th(r)),\]
while 
\[t w_h(\psi)(r)=t( w\circ \psi-\psi \circ h)(r)= w (\psi(tr))-\psi(h(tr))=
w (\psi(tr))-\psi(th(r))- \psi(rh(t))\,.  \]
Then, we prove that $(h,w_h)$ extends $(h,u)$, i.e., $\beta g u=w_h\beta g$. For every $r\in R$ and every $m\in M$, we have
\[\begin{split}
(w_h\beta g(m))(r)&= w(\beta g(m)(r)) -  \beta g(m)(h(r))= w(q(rg(m))-q(h(r)g(m))\\
&=w(q(g(rm))-q(g(h(r)m))=qgu(rm)-q(g(h(r)m))\\
&=qg(ru(m))+qg(h(r)m)-q(g(h(r)m))=qg(ru(m))=\beta g(u(m))(r)\,.\end{split}
\]
Finally, it is sufficient to take $(h,v)=(h, \gamma w_h\beta) \in D_A(R,J)$. Item (1) is now clear and item (2) follows by 
considering the injective morphism $0\to J$.
\end{proof}

\begin{corollary}\label{cor.liftingproperty} Let $0\to K\xrightarrow{\alpha}P\xrightarrow{\beta}M\to 0$ be a short exact sequence of projective $R$-modules, and denote by 
\[ L=D_A(R,K\xrightarrow{\alpha}P)=D_A(R,P\xrightarrow{\beta}M)=\{(h,u)\in D_A(R,P)\mid u(\alpha(K))\subset \alpha(K)\}\,.\]
Then, the natural morphisms of Lie algebras
\[L=D_A(R,K\xrightarrow{\alpha}P)\to  D_A(R,K),\quad L=D_A(R,K\xrightarrow{\alpha}P)\to  D_A(R,P),\]
\[  D_A(R,P\xrightarrow{\beta}M)\to D_A(R,M),\] fit in the following 
 exact sequences of $R$-modules
\[ 0\to L\to  D_A(R,P)\xrightarrow{p} \Hom_R(K,M)\to 0,\qquad
 0\to \Hom_R(P,K)\xrightarrow{j} L\to D_A(R,M)\to 0,\]
\[   0\to \Hom_R(M,P)\xrightarrow{h} L\to D_A(R,K)\to 0,\]
where $p(h,u)=\beta u\alpha$,  $j(v)=(0,\alpha v)$, $h(u)=(0,u\beta)$.
\end{corollary}

\begin{proof} Since $\beta\alpha=0$, the morphism $p$ is properly defined and it is surjective because  $M$ is projective,   $P\simeq K\oplus M$ and therefore
\[\Hom_R(P,P)\to \Hom_R(K,M),\qquad u\mapsto \beta u\alpha,\]
is surjective. The surjectivity of $L\to D_A(R,M)$ is given by Lemma~\ref{lem.liftingproperty} and therefore also  the anchor map  $\alpha\colon L\to \Der_A(R,R)$ is surjective.
The third exact sequence follows by the snake lemma applied to the commutative diagram 
\[ \xymatrix{0\ar[r]&S\ar[d]\ar[r]&L\ar[d]\ar[r]&\Der_A(R,R)\ar@{=}[d]\ar[r]&0\\
0\ar[r]&\Hom_R(K,K)\ar[r]&D_A(R,K)\ar[r]&\Der_A(R,R)\ar[r]&0}\]
where $S=\{f\in \Hom_R(P,P)\mid f(K)\subset K\}$. 
\end{proof}

Consider now a morphism of commutative unitary rings $A\to R$ and a cochain complex $C=\{\cdots\to C^i\xrightarrow{d}C^{i+1}\to\cdots\}$ of $R$ modules. Then, we can define $D_A(R,C)$ as in 
Definition~\ref{def.derivationspairs},  by replacing $\Hom_A(M,M)$ with the module of morphisms of cochain complexes of $A$-modules; equivalently $D_A(R,C)$ is defined considering $C$ as a diagram of $R$-modules over the ordered set $\Z$.

However, for the application we have in mind, it is more convenient to consider  
the DG-Lie subalgebra 
\[ D_A^*(R,C)=\left\{(h,u)\in \Der_A^*(R,R)\times \Hom_A^*(C,C)\;\middle|\;
\begin{array}{l}
 u(rx)-ru(x)=h(r)x,\\
 \text{for every } r\in R,\; x\in C\end{array}\right\}\,.\]
For the definition and the main properties of $\Der_A^*(R,R)$ and $\Hom_A^*(C,C)$ the reader may consult 
e.g., \cite[Section 1]{semireg2011}.
Notice that the differential is the internal derivation 
$\delta=[d,-]$ and $\Der^i_A(R,R)=0$ for every $i\not=0$: this implies that
$D^i_A(R,C)=\Hom_R^i(C,C)$ for every $i\not=0$, and $D^0_A(R,C)=\prod^{\times}D_A(R,C^i)$ is the limit of the diagram of anchor maps 
$D_A(R,C^i)\to \Der_A(R,R)$. Finally, note that $D_A(R,C)=Z^0(D^*_A(R,C))$.

In order to extend Lemma~\ref{lem.liftingproperty} and Corollary~\ref{cor.liftingproperty} to the differential graded case it is useful to work in the projective model structure on the category of unbounded cochain complexes \cite[Theorem 2.3.11]{Hov99}: with respect to this model structure,  a  morphism of cochain complexes $M\to N$ is a weak-equivalence if it is a quasi-isomorphism, it is a fibration if it is degreewise surjective, it is a cofibration if it 
has the left lifting property with respect to all the 
trivial fibrations.  Moreover, if $M\to N$ is a cofibration, then for every $i$ the map $M^i\to N^i$ is split injective with projective cokernel; the converse holds whenever $N$ is bounded above.

In particular, for every cofibrant complex  $P$, since $P^i$ is projective for every $i$, there exists a short exact sequence
\begin{equation}\label{equ.anchorcofibrant} 
0\to \Hom^*_R(P,P)\to D^*_A(R,P)\to \Der_A(R,R)\to 0\,.
\end{equation}

\begin{proposition}\label{prop.trivialcofib} Let $f\colon K\to P$ be a trivial cofibration between cofibrant complexes of $R$-modules. 
Then, the two natural DG-Lie algebra morphisms 
\[ D_A^*(R,K)\xleftarrow{\quad}D_A^*(R,K\xrightarrow{f}P)\xrightarrow{\quad}D_A^*(R,P)\]
are quasi-isomorphisms of complexes of $R$-modules. In particular, $D_A^*(R,K)$ and $D_A^*(R,P)$ are quasi-isomorphic DG-Lie algebras.
\end{proposition}

\begin{proof} By assumption  $f$ is a cofibration and then   
we have a short exact sequence of cofibrant complexes
$0\to K\xrightarrow{f}P\to M\to 0$.  
Since $K^i,P^i,M^i$ are projective modules for every $i$, Corollary~\ref{cor.liftingproperty} gives two short exact sequences of complexes
\[ 0\to D_A^*(R,K\xrightarrow{f}P)\to  D_A^*(R,P)\to \Hom_R^*(K,M)\to 0,\]
\[ 0\to \Hom_R^*(M,P)\to D_A^*(R,K\xrightarrow{f}P)\to D_A^*(R,K)\to 0\,.\]
Finally, since $M$ is cofibrant acyclic, the natural map $M\to  
\operatorname{Cone}(\Id_M)$ is a trivial cofibration  and then admits a left inverse; hence the 
complexes $M$, $\Hom_R^*(M,P)$ and $\Hom_R^*(K,M)$ are 
contractible.   
\end{proof}

\begin{corollary}\label{cor.homotpyinvariance} Let $P$ and $K$ be quasi-isomorphic cofibrant complexes of $R$-modules. Then, 
$D^*_A(R,P)$ is quasi-isomorphic to  $D^*_A(R,K)$ as a DG-Lie algebra.
\end{corollary}

\begin{proof} By general facts of model category theory, see e.g. \cite[Lemma 1.1.12]{Hov99}, 
two cofibrant complexes $P$ and $K$ are quasi-isomorphic if and only if there exists a span of trivial cofibrations $P\to Q\leftarrow K$. Therefore, the conclusion follows from Proposition~\ref{prop.trivialcofib}.
\end{proof}

\begin{example}\label{example successione esatta comologia per moduli}
 If $P\to M$ is a projective resolution of an $R$-module $M$, then Corollary~\ref{cor.homotpyinvariance} implies that the graded Lie algebra $H^*(D^*_A(R,P))$ depends only on $M$. Moreover, 
by \eqref{equ.anchorcofibrant} we have  
$H^i(D_A^*(R,P))=\Ext_R^i(M,M)$ for every $i\not=0,1$ and, 
by \eqref{equ.basicexactsequence} of Lemma \ref{lem.basicexactsequence}, there exists an exact sequence
\[ 0\to \Ext_R^0(M,M)\to H^0(D_A^*(R,P))\to \Der_A(R,R)\to 
\Ext_R^1(M,M)\to H^1(D_A^*(R,P))\to 0\,.\]
In fact $\Ext_R^i(M,M)=H^i(\Hom_R^*(P,M))=H^i(\Hom_R^*(P,P))$ for every $i$. 
By Lemma~\ref{lem.liftingproperty}, the natural map $Z^0(D_A^*(R,P))\to D_A(R,M)$ is surjective; its kernel is given by the $R$-linear morphisms of complexes $P\to P$ inducing the trivial map on $M$. Since $P$ is a projective resolution, these morphisms are exactly the ones homotopic to $0$, and then there exists a natural isomorphism 
$H^0(D_A^*(R,P))=D_A(R,M)$. Since $\Ext^i_R(M,M)=0$ for $i<0$, the above equality is completely 
equivalent to the long exact sequence
\[ \cdots\to \Hom^{-2}_R(P,P)\xrightarrow{\delta}\Hom^{-1}_R(P,P)\xrightarrow{\delta}
D_A(R,P)\to D_A(R,M)\to 0\,.\]
\end{example}

\bigskip
\section{Anchor invariance of Fitting ideals} 

The content of this section is not relevant for the remaining part of the paper and it is written as an application of the previous results that  we consider of independent interest.

Let $A\to R$ be a morphism of unitary commutative rings and $M$ an $R$-module.
We have seen that, unless $M$ is either projective or injective, the anchor map $D_A(R,M)\xrightarrow{\; \alpha\; } \Der_A(R,R)$ is generally not surjective.

\begin{theorem}\label{thm.fittinginvariance} 
In the above setup, if  $M$ is finitely generated, then 
for every $(h,u)\in D_A(R,M)$ the derivation  $h$ preserves the Fitting ideals of $M$.
\end{theorem}

\begin{proof} Recall (see \cite{eisenbud,Teissier}) that the Fitting ideals of $M$, $\Fitt_0(M)\subset \Fitt_1(M)\subset\cdots\subset R$, are defined by  considering any free resolution 
\[ F\xrightarrow{\;f\;}R^m\xrightarrow{\;p\;}M\to 0\]
and the exterior powers $\wedge_R^iF\xrightarrow{\;f^{\wedge i}\;}\wedge_R^iR^m$.
Then, for every $i\ge 0$, the Fitting ideal 
$\Fitt_{m-i}(M)\subset R$ is the ideal 
generated by the coefficients of all elements in the image of $f^{\wedge i}$, with respect to the canonical basis of $\wedge_R^iR^m$. The definition is independent of the choice of the resolution and then 
$\Fitt_m(M)=R$ whenever $M$ is generated by $m$ elements.  

By Lemma~\ref{lem.liftingproperty}, the derivation of pair 
$(h,u)$ lifts to a couple $(h,v)\in D_A(R,R^m)$ and
$(h,w)\in D_A(R,F)$ such that $up=pv$ and $fw=vf$.
Let us prove first that the ideal $I=\Fitt_{m-1}(M)$ is preserved by $h$. Let 
$e_1,\ldots,e_m$ be the canonical basis of $R^m$.  For every $x\in F$, we have 
\[ f(x)=\sum a_i e_i,\qquad f(w(x))=\sum b_i e_i,\qquad a_i,b_i\in I,\]
and then 
\[ \sum h(a_i)e_i=v(f(x))-\sum a_i v(e_i)=\sum b_i e_i -\sum a_i v(e_i)\in I\cdot R^m\,,\]
proving that $h(a_i)\in I$ for every $i$. 
 As regard   the invariance of the  Fitting ideals $\Fitt_{m-r}(M)$ for  $r>1$, it is sufficient to repeat the above argument to the maps $f^{\wedge r}\colon \wedge_R^rF\to \wedge_R^rR^m$ and to the Leibniz extensions of 
$(h,v)$ and $(h,w)$.
\end{proof}

\begin{corollary} Let $A\to R$ be a morphism of noetherian rings and let $M$ and $N$ be two finitely generated $R$-modules. If a derivation $h\in \Der_A(R,R)$ lifts to $D_A(R,M)$ and to $D_A(R,N)$, then 
$h$ preserves the Fitting ideals of  $\Tor_i^R(M,N)$ and $\Ext^i_R(M,N)$.
\end{corollary}

\begin{proof} Let $P\to M$ be a projective resolution, then $h$ lifts to an element of $D_A(R,P)$, and by Lemma~\ref{lem.tensor_e_hom} $h$ lifts also to 
$D_A(R,\Tor_i^R(M,N))$ and $D_A(R,\Ext^i_R(M,N))$.
\end{proof}

Joining Theorem~\ref{thm.fittinginvariance} and Example~\ref{ex.Liederivative} 
we get a new proof of the following classical result.

\begin{corollary}[R. Hart \cite{Hart}] Let $A$ be a commutative ring and $R$ a finitely generated commutative $A$-algebra. Then, every 
$A$-derivation of $R$ preserves the Fitting ideals of $\Omega_{R/A}$. 
\end{corollary}

More generally, taking the Leibniz extension of the Lie derivative we also obtain that 
every 
$A$-derivation of $R$ preserves the Fitting ideals of $\wedge_R^n\Omega_{R/A}$ for every $n$. 
Notice also that the Leibniz extension of the Lie derivative commutes with the de Rham differential.

\begin{example} Assume that $R$ is a principal ideal domain and that $M$ is a finitely generated $R$-module. 
Then, a derivation $h\in \Der_A(R,R)$ lifts to $D_A(R,M)$ if and only if $h$ preserves the Fitting ideals of $M$.
In fact, we have a cyclic decomposition 
\[ M=\frac{R}{(a_1)}e_1\oplus\cdots\oplus\frac{R}{(a_n)}e_n,\qquad a_i|a_{i+1},\]
and therefore the Fitting ideals are 
\[(a_1a_2\cdots a_n)\subset \cdots (a_1a_2)\subset (a_1)\;.\]
Assume that $h$ preserves all the Fitting ideals, then  
$h(a_i)\in (a_i)$ for every $i$; this is clear if either $i=1$ or $a_i=0$, while for $i>1$ and $a_i\not=0$ we have 
\[ h(a_1\cdots a_{i-1})\in (a_1\cdots a_{i-1}),\qquad h(a_1\cdots a_i)\in (a_1\cdots a_i),\]
\[ h(a_1\cdots a_i)=h(a_1\cdots a_{i-1})a_i+a_1\cdots a_{i-1}h(a_i),\]
 and then 
\[ a_1\cdots a_{i-1}h(a_i)\in (a_1\cdots a_{i-1}a_i)\,.\]
Now the  map
\[ u\colon M\to M,\qquad u(\sum x_ie_i)=\sum h(x_i)e_i\]
gives a derivation of pair $(h,u)$. 
\end{example}

\bigskip
\section{Coherent sheaves of DG-Lie algebras and trace maps}\label{sezione definizione traccia}

Assume now that $X\to \Spec A$ is a scheme over $A$, and denote by 
\[\Theta_{X/A}=\HOM_{\Oh_X}(\Omega_{X/A},\Oh_X)=\DER_A(\Oh_X,\Oh_X)\] 
the relative tangent sheaf. Given 
a sheaf $\sF$ of $\Oh_X$-modules,  we can define $\sD_A(X,\sF)$ as the subsheaf of $\Theta_{X/A}\times \HOM_A(\sF,\sF)$ of all the elements satisfying the same condition of Definition~\ref{def.derivationspairs}. Similarly, for every morphism $f\colon \sF\to\sG$ of
sheaves of $\Oh_X$-modules we can define the sheaf 
$\sD_A(X,\sF\xrightarrow{f}\sG)$, for instance by the exact sequence:
\[ 0\to \sD_A(X,\sF\xrightarrow{f}\sG)\to \sD_A(X,\sF)\times_{\Theta_{X/A}}\sD_A(X,\sG)\xrightarrow{\Phi}\HOM_{\Oh_X}(\sF,\sG),\quad \Phi(h,u,v)=fu-vf\,.\]

If $\sF$ is coherent and 
$p\colon Y\to X$ is the affine morphism of schemes such that $p_*\Oh_Y=\Oh_X\oplus \sF$, see e.g., 
\cite[Exercise II.5.17]{Ha77}, then by Lemma~\ref{lem.esattapercoerenza} we have an exact sequence of sheaves of $\Oh_X$-modules
\begin{equation}\label{equ.esattaperfasci} 
0\to \sD_A(X,\sF)\to p_*\Theta_{Y/A}\to \HOM_{\Oh_X}(\sF,\Oh_X)\oplus \HOM_{\Oh_X}(\Omega_{X/A},\sF)\,.\end{equation}

\begin{proposition}\label{propo D sono fasci coerenti}
In the above situation, if $A$ is noetherian, $X$ of finite type over $A$ and $\sF$ and $\sG$ coherent sheaves, then 
also  $\sD_A(X,\sF)$, $\sD_A(X,\sG)$ and $\sD_A(X,\sF\xrightarrow{f}\sG)$ are coherent. For every open affine subset $U\subset X$, we have: 
\[ \sD_A(X,\sF)(U)\cong D_A(\Oh_X(U),\sF(U)),\qquad 
\sD_A(X,\sF\xrightarrow{f}\sG)(U) \cong D_A(\Oh_X(U),\sF(U)\xrightarrow{f}\sG(U))\,.\]
\end{proposition}

\begin{proof} According to the above assumptions, the scheme $Y$ is also  of finite type over $A$ and the morphism $p$ is finite. In particular, 
both $\Omega_{X/A}$ and $p_*\Theta_{Y/A}$ are coherent sheaves and so $\sD_A(X,\sF)$ is coherent, being the kernel of a morphism of coherent sheaves. Given an open affine subset $U\subset X$, the proof that the natural map 
$\sD_A(X,\sF)(U)\to D_A(\Oh_X(U),\sF(U))$ is an isomorphism follows by comparing the exact sequences
obtained by applying the functor $\Gamma(U,-)$ to  
\eqref{equ.esattaperfasci} and  the exact sequence of Lemma~\ref{lem.esattapercoerenza}.
The  statement about $\sD_A(X,\sF\xrightarrow{f}\sG)$ is proved in the same way.
\end{proof}

\begin{example}\label{example.DL algebroid atya}
Let $\mathcal{F}$ be a locally free sheaf  on $X$, then the sheaf $\sD_A(X,\sF)$ is isomorphic to the sheaf of $A$-linear
first order differential operators on $\sF$ with scalar symbol. Via this isomorphism, the 
exact sequence 
\begin{equation}\label{eq atya sequence}
   0 \to \HOM_{\Oh_X}(\sF,\sF) \to \sD_A(X,\sF) \to  \Theta_{X/A} \to 0
\end{equation}
corresponds, up to isomorphism, to the Atiyah extension of $\sF$, cf.  \cite{Ati}, 
\cite[Example 2.3]{manetticinese}, \cite[p. 145]{Sernesi}. 
\end{example}

\begin{definition} Let $X\to \Spec A$ be a noetherian scheme  over a noetherian ring $A$. By a (quasi)coherent sheaf of Lie algebras over $X/A$ we  mean a (quasi)coherent sheaf of $\Oh_X$-modules $\sL$, together an $A$-bilinear bracket 
$\sL\times \sL\xrightarrow{\;[-,-]\:} \sL$ inducing a structure of Lie algebra over $A$ on every stalk of $\sL$.  Note that the bracket is not  $\Oh_X$- linear, so this is not the same structure as a Lie algebra in the monoidal category of quasi coherent module.
A morphism of  (quasi)coherent sheaves of Lie algebras is a morphism of sheaves of $\Oh_X$-modules commuting with the brackets.
\end{definition}

For instance,  if $\sF$ is a coherent sheaf on $X$, then the anchor map 
$\alpha\colon \sD_A(X,\sF)\to \Theta_{X/A}$ is a morphism of coherent sheaves of Lie algebras over $X/A$.
If $\sF\to \sG$ is a morphism of coherent sheaves, then the natural maps 
\[ \sD_A(X,\sF\to\sG)\to \sD_A(X,\sF),\qquad \sD_A(X,\sF\to\sG)\to \sD_A(X,\sG),\]
are morphisms of coherent sheaves of Lie algebras over $X/A$.

Recall that a differential graded (DG) Lie algebra over a commutative ring $A$ is 
the data of a cochain complex $(L,d)$ of $A$-modules,  together
with an   $A$-bilinear map $[- ,-] \colon L \times L \to L$  (called bracket)
of degree 0, such that the following conditions are satisfied:
\begin{enumerate}

\item (graded  skewsymmetry)
$[x,y]=-(-1)^{ij} [y,x]\in L^{i+j}$,  for every $x\in L^i$ and $y\in L^j$;

\item $[x,x]=0$ for every $x\in L^{2i}$, and $[x,[x,x]]=0$ for every $x\in L^{2i+1}$, with $i\in \Z$;

\item (graded Jacobi identity) $ [x,[y,z]] = [[x,y],z] + (-1)^{ij} [y, [x,z]]$, for every $x\in L^i$, $y\in L^j$ and $z\in L$;

\item (graded Leibniz rule) $ d[x,y] =[ dx,y]+ (-1)^{i}[x, dy]$, for every $x\in L^i$ and $y\in L^j$.

\end{enumerate}

A morphism of differential graded Lie algebras $\chi \colon L \to M$ is a morphism of cochain complexes that commutes with brackets.

\begin{definition} Let $X\to \Spec A$ be a noetherian scheme  over a noetherian ring $A$. By a quasi-coherent sheaf of DG-Lie algebras over $X/A$ we  mean a  complex $\sL^*$ of quasi-coherent sheaves 
of $\Oh_X$-modules, together an $A$-bilinear bracket 
$\sL^*\times \sL^*\xrightarrow{\;[-,-]\:} \sL^*$ inducing a structure of DG-Lie algebra over $A$ on every stalk of $\sL$. A quasi-coherent sheaf of DG-Lie algebras $\sL^*$ is coherent if $\oplus \sL^i$ is a coherent sheaf.
A morphism of  quasi-coherent sheaves of DG-Lie algebras is a morphism of complexes  of $\Oh_X$-modules commuting with the brackets. 
\end{definition}

\begin{example}\label{example algebroid complex of sheaf}
Let $X\to \Spec A$ be a noetherian scheme  over a noetherian ring $A$, and let 
$\sE^*$ be a bounded complex of coherent sheaves of $\Oh_X$-modules. For every index $i$, the sheaf  
\[ \HOM_{\Oh_X}^i(\sE^*,\sE^{*})=\prod_j \HOM_{\Oh_X}(\sE^j,\sE^{j+i})\]
is coherent and then $\HOM^*_{\Oh_X}(\sE^*,\sE^*)$ is a coherent sheaf of DG-Lie algebras.
  
As in Section~\ref{sec.lifting}, we can define the complex of $\Oh_X$-modules $\sD^*_A(X,\sE^*)$ as the subsheaf 
of  $\DER^*_A(\Oh_X, \Oh_X) \times \HOM^*_A(\sE^*,\sE^*)$, whose elements are the pairs $(h,u)$ such that 
\[
u(rm)-ru(m)=h(r)m, \qquad \text{for every } r\in \Oh_X,\; m\in \sE^*.
\]
For every $i\not=0$ we have $\sD^i_A(X,\sE^*)=\HOM_{\Oh_X}^i(\sE^*,\sE^{*})$, while $\sD^0_A(X,\sE^*)$ is the limit of the 
set of anchor maps $\sD^0_A(X,\sE^j)\xrightarrow{\alpha}\DER_A(\Oh_X,\Oh_X)=\Theta_{X/A}$. Thus 
$\sD^*_A(X,\sE^*)$ is a coherent sheaf of DG-Lie algebras and there exists a short exact sequence   of 
coherent sheaves of DG-Lie algebras over $X/A$:
\begin{equation}\label{eq atyah sequence for complex}
 0 \to   \HOM^*_{\Oh_X}(\sE^*,\sE^*) \to \sD^*_A(X,\sE^*) \xrightarrow{\;\alpha\;} \Theta_{X/A}\,.
\end{equation}
According to Lemma~\ref{lem.liftingproperty}, the anchor map $\sD^*_A(X,\sE^*) \xrightarrow{\;\alpha\;} \Theta_{X/A}$ is surjective whenever 
every sheaf $\sE^j$ is locally free. Analogously to Remark \ref{propo D sono fasci coerenti},  for every open affine subset $U\subset X$, we have: 
\[ \sD_A(X,\sE^*)(U)\cong D_A(\Oh_X(U),\sE^*(U))\,.\]
\end{example}

It is well known that for every bounded complex  
$\sE^*$ of locally free sheaves we can define the trace map 
\begin{equation}\label{equ.traccia degreenot0} 
\Tr\colon \HOM^*_{\Oh_X}(\sE^*,\sE^*)\to \Oh_X,
\end{equation}
which is a surjective morphism of complexes of coherent sheaves: given  $f\in \HOM^i_{\Oh_X}(\sE^*,\sE^*)$ then
$\Tr(f)=0$ for $i\not=0$; if $i=0$, then $f$ is the datum of a finite sequence of morphisms of locally free sheaves $f_j\colon \sE^j\to \sE^j$, and 
\[ \Tr(f)=\sum_j (-1)^j\Tr(f_j),\]
where $\Tr\colon \HOM_{\Oh_X}(\sE^j,\sE^j)\to \Oh_X$ is the usual trace map.

It is easy to verify that $\Tr$ is also a morphism of sheaves of DG-Lie algebras, where $\Oh_X$ is equipped with the trivial bracket. In fact, since the differential on
$\HOM^*_{\Oh_X}(\sE^*,\sE^*)$ is an adjoint operator, it is sufficient to prove that  
$\Tr([f,g])=0$ for every $f\in \HOM^i_{\Oh_X}(\sE^*,\sE^*)$ and $g\in \HOM^j_{\Oh_X}(\sE^*,\sE^*)$. 
This is clear if $i+j\not=0$, and so it is not restrictive to assume $i=-j$. Since the trace is 
$\Oh_X$-linear, we can reduce this fact to the following statement of linear algebra: let $V_i$, $i\in \Z$, be a sequence  of finite dimensional vector spaces, with $V_i\not=0$ for at most finitely many indices $i$.  Then for every integer $p$ and any two sequences 
\[ f_i\colon V_i\to V_{i+p},\qquad   g_i\colon V_i\to V_{i-p},\]
of linear maps, we have   
\[ \sum_i (-1)^i\Tr\left(f_{i-p}g_{i}-(-1)^{p^2}g_{i+p}f_i\right)=0\,.\]
In fact, by a basic fact of linear algebra we have $\Tr(f_{i-p}g_{i})=\Tr(g_if_{i-p})$ for every $i$, and then
\[\begin{split} 
\sum_i (-1)^i\Tr\left(f_{i-p}g_{i}-(-1)^{p^2}g_{i+p}f_i\right)&=\sum_i (-1)^i
\Tr(f_{i-p}g_{i})-(-1)^{i+p}\Tr(g_{i+p}f_i)\\
&=\sum_i (-1)^i
\Tr(f_{i-p}g_{i})-(-1)^{i}\Tr(g_{i}f_{i-p})=0\,.\end{split}\]

Denoting by $\HOM^*_{\Oh_X}(\sE^*,\sE^*)_0$ the sheaf of traceless endomorphisms of $\sE^*$, i.e., the kernel of \eqref{equ.traccia degreenot0}, and by $\Ext^i_{X}(\sF,\sF)_0$ its hypercohomology groups, 
there exists a long exact sequence 
\begin{equation}\label{equ.ext0} \to \Ext^1_{X}(\sF,\sF)\xrightarrow{\Tr}H^1(X,\Oh_X)\to 
\Ext^2_{X}(\sF,\sF)_0\to \Ext^2_{X}(\sF,\sF)\xrightarrow{\Tr}H^2(X,\Oh_X)\to\;.\end{equation}

By the results of Section~\ref{sec.derivationpairs}, we can extend the trace map 
$\Tr\colon \HOM^0_{\Oh_X}(\sE^*,\sE^*)\to \Oh_X$ 
to a Lie algebra morphism 
\begin{equation}\label{equ.traccia degree0} 
\Tr\colon \sD_A^0(X,\sE^*)\to \sD_A(X,\det \sE^*)\,.
\end{equation}
Assume that the finite complex $\sE^*$ is nonzero only for degree between $n$ and $m$, then  
by definition,  an element of $\sD_A^0(X,\sE^*)$ is given by a 
sequence
$(h,u_n,\ldots,u_m)$ where $h\in \DER_{A}(\Oh_X,\Oh_X)$ and 
$(h,u_i)\in \sD_A(X,\sE^i)$ for every $i$. 
Taking the trace of every pair $(h,u_i)$ (Definition~\ref{def.tracemodule}) we obtain a finite sequence of derivations of pairs 
\[ (h,v_i)\in \sD_A(X,\det \sE^i),\qquad v_i=\Tr(u_i),\qquad i=n,\ldots,m\;.\] 
Then, for every odd index $i$ we consider the transpose \eqref{equ.transposition} 
\[ (h,v_i)^T\in \sD_A(X,(\det \sE^i)^{-1})\,.\]
Since 
\[ [(h,v_0),(h,v_1)^T]\in \sD_A(X,\det \sE^0)\times_{\DER_{A}(\Oh_X,\Oh_X)}\sD_A(X,(\det \sE^1)^{-1})\]
we may apply the Lie morphism $\Phi$ of  Lemma~\ref{lem.tensor_e_hom} in order to obtain 
an element  $\Tr(h,u_0,u_1)\in \sD_A(X,(\det\sE^0)\otimes (\det \sE^1)^{-1})$. 
It is now clear that, with a finite number of  constructions as above, we have 
a well defined element 
\[ \Tr(h,u_n,\ldots,u_m)\in \sD_A(X,\det \sE^*)=\sD_A\left(X,\bigotimes_{i=n}^m(\det \sE^i)^{(-1)^i}\right)\,.\]   
Moreover, we have also proved that all the above considered maps
\[ \sD_A(X,\sE^i)\to \sD_A(X,\det \sE^i),\qquad {\prod_i}^{\times}\sD_A(X,\det \sE^i)\to 
\sD_A(X,\det \sE^*),\]
are morphisms of sheaves of Lie algebras. Keeping in mind that 
\[ \sD_A^0(X,\sE^*)={\prod_i}^\times\sD_A(X,\sE^i),\qquad \sD_A^j(X,\sE^*)=\HOM^j_{\Oh_X}(\sE^*,\sE^*),\quad j\not=0,\]
we can glue together \eqref{equ.traccia degree0} and \eqref{equ.traccia degreenot0} and obtain a morphism of sheaves of DG-Lie algebras
\begin{equation}\label{equ.traccia alldegree} 
\Tr\colon \sD_A^*(X,\sE^*)\to \sD_A(X,\det \sE^*)\,.
\end{equation}

The following theorem is now clear.

\begin{theorem}\label{thm.tracediagram} For every bounded complex of locally free sheaves $\sE^*$ on a scheme $X$ of finite type over a noetherian ring $A$, there exists a commutative diagram of morphisms of coherent sheaves of DG-Lie algebras
\[ \xymatrix{0\ar[r]&\HOM^*_{\Oh_X}(\sE^*,\sE^*)\ar[d]^{\Tr}\ar[r]&\sD^*_A(X,\sE^*)\ar[d]^{\Tr}\ar[r]^-{\alpha}&\Theta_{X/A}\ar[r]\ar@{=}[d]&0\;\;\\
0\ar[r]&\Oh_X\ar[r]&\sD_A(X,\det \sE^*)\ar[r]^-{\alpha}&\Theta_{X/A}\ar[r]&0\;.}\]
\end{theorem}
   
\bigskip

\section{A short review of deformation theory via DG-Lie algebras}\label{sezione DGLA e tot}
\label{sec.review}

The main references for this section are \cite{FIM,GoMil1,ManRendiconti,ManettiSemireg,ManettiSeattle}.
From this section, and throughout the rest of the paper, we work over a fixed algebraically closed  field $\K$ of characteristic zero.
We denote by $\Set$  the category of sets 
and by $\Art_\K$   the category of Artin local
$\K$-algebras with residue field $\K$, and by  $\mathbf{DGLA}=\mathbf{DGLA}_{\K}$  the category of DG-Lie algebras over $\K$. Unless otherwise specified,
for every  local algebra  $A\in \mathbf{Art}_{\K}$, we denote by
$\mathfrak{m}_A$ its maximal ideal.

Given a DG-Lie algebra  $L$ over $\K$,  we  can define two functors of Artin rings.
The Maurer-Cartan functor $\MC_L \colon \mathbf{Art}_{\K}\to
\mathbf{Set}$ is defined by:
\[
\MC_L(A)=\left\{ x \in L^1\otimes \mathfrak{m}_A \mid dx+ \displaystyle\frac{1}{2}
[x,x]=0 \right\},
\]
where the DG-Lie structure on $L \otimes \mathfrak{m}_A$ is obtained by scalar extension from the 
DG-Lie structure on $L$.
The deformation functor $\Def_L\colon\Art_\K \longrightarrow \Set$ is:
\[
\Def_L(A)=\frac{\MC_L(A)}{\text{gauge}}=\frac{\{ x \in L^1\otimes \mathfrak{m}_A \ |\ dx+
\displaystyle\frac{1}{2} [x,x]=0 \}}{\exp(L^0\otimes \mathfrak{m}_A )  }\;,
\]
where the gauge action $\ast\colon\exp(L^0 \otimes \mathfrak{m}_A)\times
\MC_L(A)\longrightarrow {\MC}_L(A)$ may be defined by the explicit formula
\[
e^a \ast x:=x+\sum_{n\geq 0} \frac{ [a,-]^n}{(n+1)!}([a,x]-da)\,.
\]

A quasi-isomorphism of DG-Lie algebras is a morphism
that induces an isomorphism in cohomology.  Two DG-Lie algebras $L$ and $M$ are said to be 
quasi-isomorphic, or homotopy equivalent, if they are equivalent under the
equivalence relation generated by quasi-isomorphisms.

By standard facts about deformation functors 
\cite{ManettiSemireg,ManettiSeattle},  it is known that the tangent space to $\Def_L$ is isomorphic to
$H^1(L)$ and that $H^2(L)$  is an obstruction space.

\begin{remark}\label{rem.basictheorem}
Every morphism $L\to M$ of DG-Lie algebras  induces a natural transformation of the
associated deformation functors $\Def_L\to \Def_M$. A basic result \cite[Thoerema 1.5]{ManettiSeattle} asserts that if
$H^0(L)\to H^0(M)$ is surjective, $H^1(L)\to H^1(M)$ is bijective and 
$H^2(L)\to H^2(M)$ is injective, then $\Def_L\to \Def_M$ is an isomorphism. 
\end{remark}

\begin{lemma}\label{lem.criteriouniversal} 
Let $L$ be a DG-Lie algebra over a field of characteristic 0 with $H^1(L)$ finite dimensional. If the natural map 
\[ N^0=\{x\in L^0\mid dx=0,\; [x,L^1]=0\}\to H^0(L)\]
is surjective, then $\Def_L$ is pro-representable.
\end{lemma}

\begin{proof} Notice first that $N^0$ is a Lie subalgebra of $L^0$. Then 
define $N^i=L^i$ for every $i\ge 2$ and choose a vector subspace
$N^1\subset L^1$ such that $L^1=d(L^0)\oplus N^1$.
The inclusion of DG-Lie algebras $N=\oplus_{i\ge 0}N^i\hookrightarrow L$ satisfies the condition of Remark~\ref{rem.basictheorem} and then $\Def_L=\Def_N$. On the other side the gauge action on $\MC_N$ is trivial and then 
$\Def_N=\MC_N$ is pro-representable by Schlessinger's theorem \cite[Thm. 2.11]{Sch}.
\end{proof}

Let $\mathbf{\Delta}_{\operatorname{mon}}$ be  the category whose objects are finite
ordinal sets and whose morphisms are order-preserving injective
maps between them.  
A  semicosimplicial differential graded Lie algebra is a
covariant functor $\mathbf{\Delta}_{\operatorname{mon}}\to
\mathbf{DGLA}$. Equivalently, a
semicosimplicial DG-Lie algebra $\mathfrak{g}^\Delta$ is a diagram
 \[
\xymatrix{ {{\mathfrak g}_0}
\ar@<2pt>[r]\ar@<-2pt>[r] & { {\mathfrak g}_1}
      \ar@<4pt>[r] \ar[r] \ar@<-4pt>[r] & { {\mathfrak g}_2}
\ar@<6pt>[r] \ar@<2pt>[r] \ar@<-2pt>[r] \ar@<-6pt>[r]&
\cdots},
\]
where each  ${\mathfrak g}_i$ is a DG-Lie algebra, and for each
$ i > 0 $, there are $ i + 1$ morphisms of DG-Lie algebras
\[
\partial_{k,i} \colon  {\mathfrak g}_{i-1}\to {\mathfrak
g}_{i},
\qquad k=0,\dots,i,
\]
such that $\partial_{ k+1, i+1} \partial_{l , i}= \partial_{l,i+1}\partial_{k,i}$,
for any  $k\geq l$. 
Here we use the non-standard notation of lower indexes for a semicosimplicial object, since the upper indexes are already used to denote degrees.

We denote by $\Tot( {\mathfrak g}^\Delta)$ the image of a  semicosimplicial differential graded Lie algebra ${\mathfrak g}^\Delta$ 
 via the  Thom-Whitney totalization functor 
\[\Tot\colon\mathbf{DGLA}^{\Delta_{\operatorname{mon}}}\to
\mathbf{DGLA}.
\] 
We refer to \cite[Section2]{Rugg}, \cite[Section 3.2]{FIM},
\cite[Section 3]{FMM} and \cite[Section 3]{semireg2011} for an explicit description of the DG-Lie algebra ${{\Tot}(\g^{\Delta})} $ and its properties. Here, we remind only  that  there exists a quasi-isomorphism of  complexes   of vector spaces between ${{\Tot}(\g^{\Delta})}$ and the total complex associated with the  cochain complex $C(\g^{\Delta})$ (considering any $\g^{\Delta}$ as a semicosimplicial object in the abelian category of DG-vector spaces). As a consequence of this fact, if $f\colon \g^{\Delta}\to \mathfrak{h}^{\Delta}$ is a morphism of  semicosimplicial differential graded Lie algebras such that $f\colon \g_i\to \mathfrak{h}_i$ is a quasi-isomorphism of DG-Lie algebras for every $i$, then the image 
$\Tot(f)\colon\Tot(\mathfrak{g}^{\Delta}) \to \Tot(\mathfrak{h}^{\Delta})$ is a quasi-isomorphism.

Considering the DG-Lie algebra  ${{\Tot}(\g^{\Delta})} $, we can naturally associate with any  semicosimplicial DG-Lie algebra $\g^\Delta$ the  functor  of Artin rings $\Def_{{\Tot}(\g^{\Delta})} \colon \Art_\K \to \Set$.
According to  \cite[Definition 3.1 and Definition 3.3]{FIM}, we can also associate with  $\g^\Delta$ two other functors of Artin rings. The former 
\[Z^1_{\rm sc} (\exp \g^{\Delta})\colon
\mathbf{Art}_{\mathbb K} \to \mathbf{Set}\]
 is defined, for any
$A\in \mathbf{Art}_{\mathbb K}$, by
\[ Z^1_{\rm sc}(\exp \g^{\Delta})(A)= \left\{  (l,m)\in
(\g_0^1 \times \g^0_1) \otimes \mathfrak m_A  \,\middle|
\begin{array}{l}
dl+ \frac{1}{2}[l,l]=0,\\  \partial_{1,1}l=e^{m}*\partial_{0,1}l, \\
{ \partial_{0,2}m} \bullet  {-\partial_{1,2}m} \bullet
{\partial_{2,2}m} =dn+[\partial_{2,2}\partial_{0,1}l,n]\\
\qquad\qquad \qquad \qquad \text{ for some $n\in {\mathfrak g}_2^{-1} \otimes{\mathfrak m}_A$}
\end{array} \right\}.\]
The latter
\[ H^1_{\rm sc}(\exp \g^\Delta) \colon \mathbf{Art}_{\mathbb K}
 \to \mathbf{Set}\]
  is defined, for any
$A\in \mathbf{Art}_{\mathbb K}$, by
\[H^1_{\rm sc}(\exp \g^\Delta)(A)= {Z^1_{\rm sc}
(\exp \g^\Delta)(A)}/{\sim}\;,\]
where two elements
$(l_0,m_0)$ and $(l_1,m_1) \in Z^1_{\rm sc} (\exp \g^{\Delta})(A)$
are equivalent under the relation  $\sim$ if  and only if there
exist elements $a \in {\g^0}_0\otimes {\mathfrak m}_A$ and $b\in{\mathfrak
g}_1^{-1}\otimes{\mathfrak m}_A$ such that
\[
\begin{cases}
e^a * l_0=l_1\\
- m_0\bullet -\partial_{1,1}a \bullet m_1
\bullet \partial_{0,1}a=db+[\partial_{0,1}l_0,b]\,.
\end{cases}
\] 
\begin{theorem}
Let ${\mathfrak g}^\Delta$ be a semicosimplicial DG-Lie algebra such that $H^{j}(\g_i)=0$ for all $i\geq 0$ and $j<0$. Then, there exists an   isomorphism of functors of Artin rings
\begin{equation}\label{eq iso funtori Def H1}
\Def_{{\Tot}(\g^{\Delta})} \to H^1_{\rm sc}(\exp\g^\Delta)\,.
\end{equation}
\end{theorem}

\begin{proof}
See \cite[Theorem 4.10]{FIM}.
\end{proof}

\begin{remark}\label{remark funtore lie algebre esplicito}
If each $\g_i$ is concentrated in degree zero, i.e.,
$\g^\Delta$ is a semicosimplicial  Lie
algebra,  then the functor $H^1_{\rm sc}(\exp\g^\Delta)$ has an easier   explicit description \cite[Section 1]{FMM}:
\[
H^1_{\rm sc}(\exp \g^\Delta)\colon \Art_\K \to \Set
\]
\[
H ^1_{sc}(\exp { \g}^\Delta )(A)=\frac{\{ x \in {\g}_1 \otimes
\mathfrak{m}_A   \ |\  e^{\de_{0}x}e^{-\de_{1}x}e^{\de_{2}x}=1
\}}{\sim},
\]
where $x \sim y$ if and only if there exists
$a\in {\g}_0\otimes\mathfrak{m}_A$, such that
$e^{-\de_{1}a}e^{x}e^{\de_{0}a}=e^y$.

\end{remark}

Let $\sL$ be a coherent sheaf of DG-Lie algebras over an algebraic variety $X$ and  $\mathcal{U}=\{U_i\}$  an  affine open cover of $X$. In this case, the \v{C}ech (double) complex $C(\mathcal{U}, \sL)$ of $\sL$ is exactly the total cochain complex  associated with 
the semicosimplicial DG-Lie algebra:
\[ \sL(\mathcal{U}):\quad \xymatrix{ {\prod_i\mathcal{L}(U_i)}
\ar@<2pt>[r]\ar@<-2pt>[r] & { \prod_{i,j}\mathcal{L}(U_{ij})}
      \ar@<4pt>[r] \ar[r] \ar@<-4pt>[r] &
      {\prod_{i,j,k}\mathcal{L}(U_{ijk})}
\ar@<6pt>[r] \ar@<2pt>[r] \ar@<-2pt>[r] \ar@<-6pt>[r]& \cdots},\]
where the face operators  $  \displaystyle \partial_{h}\colon
{\prod_{i_0,\ldots ,i_{k-1}} \sL(U_{i_0 \cdots  i_{k-1}})}\to
{\prod_{i_0,\ldots ,i_k} \sL(U_{i_0 \cdots  i_k})}$
are given by
\[\partial_{h}(x)_{i_0 \ldots i_{k}}={x_{i_0 \ldots
\widehat{i_h} \ldots i_{k}}}_{|U_{i_0 \cdots  i_k}},\qquad
\text{for }h=0,\ldots, k.\]

For notational convenience, we denote by $\Tot(\mathcal{U}, \sL)$ the Thom-Whitney totalization  of the semicosimplicial DG-Lie algebra $\sL(\mathcal{U})$.
Note that in this case 
\begin{equation}\label{eq. comologia tot is hypercoho}
H^i(\Tot(\mathcal{U}, \sL))= H^i(
C(\mathcal{U}, \sL))=\check{\H}^i(\mathcal{U},\sL) = \H^i (X,\sL)\, .  
\end{equation}

\begin{remark} For any   coherent sheaf of DG-Lie algebras  $\sL$,    the quasi-isomorphism class of $\Tot(\mathcal{U}, \sL)$ does not depend on the choice of the affine open cover \cite{FIM}.
If $\sL \to \sG$ is a morphism of coherent sheaves of DG-Lie algebras that is a quasi-isomorphism at every stalk, then for any affine open cover $\mathcal{U}$, the induced morphism of 
DG-Lie algebras  $\Tot(\mathcal{U}, \sL) \to \Tot(\mathcal{U}, \sG)$ is a quasi-isomorphism of DG-Lie algebras.
\end{remark}
 This motivates the following definition.

\begin{definition}\label{definition fascio L controlla}
Let $\sL$  be a   (quasi) coherent sheaf   of DG-Lie algebras over an algebraic variety $X$. Suppose that, for an open affine cover  $\mathcal{U}$ of $X$ (and so for all), the  DG-Lie algebra $\Tot(\mathcal{U}, \sL)$ controls a given deformation problem, then 
we also  say that $\sL$  controls the deformation problem.
Note that in this case   $\H^1 (X,\sL)$ is the tangent space and $\H^2 (X,\sL)$ is an obstructions  space for the controlled deformation problem.

\end{definition}

\bigskip
\section{Deformations of pairs (scheme, coherent sheaf)}\label{sezione deformazioni coppie geoemtrica}

From now on, unless otherwise specified,  we  denote by $X$ a scheme of finite type over the field $\K$.

\begin{definition}
Let   $ \sF$  be a coherent sheaf on $X$. An infinitesimal deformation of the pair $(X, \sF)$ over  $A\in \Art_\K$ is the data $(\mathcal{X}_A, \sF_A)$ where:
\begin{itemize}
\item $\mathcal{X}_A$ is an infinitesimal deformation  of $X$ over $A$, i.e., a pull-back  diagram
\begin{center}
$\xymatrix{ X \ar[r] \ar[d] & \mathcal{X}_A  \ar[d]^\pi  \\
            \Spec \K \ar[r]  &     \Spec A,\\ }$
\end{center}
where $\pi$ is flat;
\item $ \sF_A$ is a coherent sheaf of $\Oh_{\mathcal{X}_A}$-modules on $\mathcal{X}_A$, flat over $A$ and a morphism $\sF_A \to \sF$, inducing an isomorphism $\sF_A \otimes_{\Oh_{\mathcal{X}_A}} \Oh_\mathcal{X} \to \sF$.
\end{itemize}
\end{definition}

 \begin{definition}
Let   $ \sF$ be a coherent sheaf on a scheme $X$.  
Two infinitesimal deformations $(\mathcal{X}_A, \sF_A)$  and $(\mathcal{X}_A', \sF_A')$ 
of  the pair $(X, \sF)$ over $A$ are isomorphic, if there exist  an isomorphism of
deformations $f\colon \mathcal{X}_A \to \mathcal{X}_A'$  and an isomorphism $  \sF_A \to f^*  \sF_A'$ of
coherent sheaves of $\Oh_{\mathcal{X}_A}$-modules.
 
\end{definition}

We recall that the trivial infinitesimal deformation of $(X, \sF)$ over  $A\in \Art_\K$ is given by the pair $(X \times \Spec A, \sF\otimes_{\Oh_\mathcal{X}}\Oh_\mathcal{X \times \Spec A}=\sF\otimes_\K A )$.

An infinitesimal deformation  of $(X, \sF)$ is locally trivial, if it is locally, in $X$, isomorphic to the trivial infinitesimal deformation.

\begin{definition}
Let  $\sF$ be a coherent sheaf on  a scheme $X$.
The functor of infinitesimal deformations   of the pair  $(X, \sF)$ is the  functor
$$
\Def_{(X,\sF)} : \Art_\K \to \Set,
$$
$$
\Def_{(X,\sF)}(A)= \{
 \mbox{isomorphism classes of deformations of the pair $(X, \sF)$  over
 $A$} \}.
$$
The 
 functor of locally trivial infinitesimal  deformations   of the pair  $(X, \sF)$ is the  functor
$$
\Def^{\text{ lt}}_{(X,\sF)} : \Art_\K \to \Set,
$$
$$
\Def^{\text{ lt}}_{(X,\sF)}(A)= \{
 \mbox{isomorphism classes of locally trivial deformations of the pair $(X, \sF)$  over
 $A$} \}.
$$

\end{definition} 
  
\begin{remark}
If $X$ is a smooth variety, then any infinitesimal deformation of $X$ is locally trivial. If $\sF$ is a locally free sheaf over a variety $X$ then any infinitesimal deformation of the sheaf $\sf$ is locally trivial. Then, there exists an isomorphism of functors $\Def^{\text{ lt}}_{(X,\sF)}\cong \Def_{(X,\sF)}$,  for a locally free sheaf $\sF$ on a smooth variety $X$.

\end{remark}

\subsection{Locally trivial infinitesimal deformations}

Let us analyse the data that define a locally trivial infinitesimal deformation $(\mathcal{X}_A, \sF_A)$ of $(X, \sF)$  over $A$. Let  $\mathcal{U}=\{U_i\}_{i\in I}$  be an affine open cover of $X$. 
The deformation $\sX_A$ of $X$ is  trivial over every affine open subset; thus for any $i \in I$ there exists an isomorphism
\[
\alpha_i \colon \Oh_{X}(U_i)\otimes A \to   \Oh_{\mathcal{X}_A}(U_i). \]
On every $U_i\times \Spec A$, we have the sheaf $\sF_{|U_i} \otimes A$  of $\Oh_{X}(U_i)\otimes A$-modules. Since $\alpha_i$ is an isomorphism, we can view $\sF_{|U_i} \otimes A$ as a sheaf of $ \Oh_{\mathcal{X}_A}(U_i)$-modules; the 
 $ \Oh_{\mathcal{X}_A}(U_i)$-modules structure on $\sF_{|U_i} \otimes A$ is given by:  $s\cdot x={\alpha_i}^{-1}(s)x$, for any $s \in  \Oh_{\mathcal{X}_A}(U_i)$ and $x\in \sF_{|U_i} \otimes A$.

To give a locally trivial deformation $\sF_A$ of the sheaf $\sF$ over $X_A$, we need to glue the sheaves  $\sF_{|U_i} \otimes A$ of $ \Oh_{\mathcal{X}_A}(U_i)$-modules over double intersections.
The isomorphisms
\[
{\alpha_i}_{|U_{ij}} \colon \Oh_{X}(U_{ij})\otimes A \to   \Oh_{\mathcal{X}_A}(U_{ij})  \qquad  \mbox{ and }\qquad {\alpha_j}_{|U_{ij}} \colon \Oh_{ X}(U_{ij})\otimes A \to   \Oh_{\mathcal{X}_A}(U_{ij}) 
\]
induce a structure of sheaf  of $\Oh_{\mathcal{X}_A}(U_{ij}) $-module  on the sheaves ${\sF_{|U_i} \otimes A}_{|U_{ij}}$ and ${\sF_{|U_i} \otimes A}_{|U_{ij}}$, respectively.
To glue these sheaves, we need an isomorphism $\psi_{ij}: {\sF_{|U_i} \otimes A}_{|U_{ij}} \to {\sF_{|U_i} \otimes A}_{|U_{ij}}$ of $\Oh_{\mathcal{X}_A}(U_{ij}) $-modules, i.e., for all $s \in \Oh_{\mathcal{X}_A}(U_{ij}) $, $\psi_{ij} (s\cdot x)= s \cdot \psi(x)$.
By the   explicit definition of $\Oh_{\mathcal{X}_A}(U_{ij}) $-modules structures, we have
\[
\psi_{ij} (s\cdot x)= \psi_{ij} ( {\alpha_i}_{|U_{ij}}^{-1}(s) x) \qquad  \mbox{ and }\qquad s \cdot  \psi_{ij}(x)=
 {\alpha_j}_{|U_{ij}}^{-1}(s) \psi_{ij} (x).
\]
Since ${\alpha_i}_{|U_{ij}}$ is an isomorphism, there exists $t \in \Oh_{X}(U_{ij})\otimes A$ such that $ s=  {\alpha_i}_{|U_{ij}}(t)$. Therefore, the linearity reads as follows 
\[
\psi_{ij} ( t x)=
 {\alpha_j}_{|U_{ij}}^{-1} {\alpha_i}_{|U_{ij}}(t) \psi_{ij} (x),
\]
for any $t \in \Oh_{X}(U_{ij})\otimes A$ and $x \in \sF_{|U_{ij}} \otimes A$.

In conclusion, in order to   define a locally trivial infinitesimal deformation of $(X, \sF)$  over $A$, it is enough to give, for any $U_{ij}$,  an automorphism $\theta_{ij}={\alpha_j}^{-1}_{|U_{ij}} {\alpha_i}_{|U_{ij}}$ of $ \Oh_X(U_{ij})\otimes A$, satisfying the cocycle condition on the triple intersections, an automorphism  $\psi_{ij} \in \Hom_\K(\sF_{|U_{ij}} \otimes A, \sF_{|U_{ij}} \otimes A)$ satisfying the cocycle condition on the triple intersections and such that $\psi_{ij} ( t x)=\theta_{ij}(t) (x)$,
for any $t \in \Oh_{X}(U_{ij})\otimes A$ and $x \in \sF_{|U_{ij}} \otimes A$.

Since we are in characteristic zero, according to Lemma \ref{lemma exp auto}, we  can take the logarithm  
and conclude that  $(\theta_{ij}, \psi_{ij}) =e^{x_{ij}}$,  where
$x_{ij}\in\sD_\K(X,\sF)(U_{ij})\otimes\mathfrak{m}_A$; the
condition of gluing on triple intersections is equivalent to
\[
e^{x_{jk}}e^{-x_{ik}}e^{x_{ij}}=1 \in
 \exp(\sD_\K(X,\sF)(U_{ijk})\otimes\mathfrak{m}_A),
\qquad \forall \ i,j,k \in I.
\]

\bigskip
As regard the isomorphisms,  suppose that  $(\mathcal{X}_A, \sF_A)$  and $(\mathcal{X}_A', \sF_A')$ are locally trivial isomorphic infinitesimal deformations of  the pair $(X, \sF)$ over $A$. Then, there exist  an isomorphism of deformations $f\colon \mathcal{X}_A \to \mathcal{X}_A'$  and an isomorphism 
$ \phi\colon \sF_A \to f^*  \sF_A'$ of
coherent sheaves of $\Oh_{\mathcal{X}_A}$-modules.

Let  $\mathcal{U}=\{U_i\}_{i\in I}$  be an affine open cover of $X$. 
Then, for each $i \in I$, there exists an induced automorphism of $\Oh_{X}(U_i)\otimes A$:
\[b_i={\alpha'_i}^{-1}   f^{-1}_{| U_i} \alpha_i:\]
\[
\Oh_{X}(U_i)\otimes A
\xrightarrow{\;\alpha_i\;}\Oh_{\mathcal{X}_A}(U_i)
\xrightarrow{\,f^{-1}_{|U_i}\,}\Oh_{\mathcal{X'}_A}(U_i)
\xrightarrow{\,{\alpha'_i}^{-1}\,}\Oh_{X}(U_i)\otimes A.\]
Therefore, ${\alpha'_i} b_i= f^{-1}_{| U_i} \alpha_i$
and so on $U_{ij}$ we have
$$
({\alpha'_i} b_i)^{-1}({\alpha'_j} b_j) _{| U_{ij}} ={\alpha_i}^{-1}
 f_{| U_{ij}}   f^{-1}_{| U_{ij}} \alpha_j={{\alpha_i}^{-1}
 \alpha_j }_{| U_{ij}};
$$
this implies
$$
{b_i}^{-1} {\theta_{ji}'}b_j= \theta_{ji} \quad \mbox{ i.e.,  } \quad \theta_{ij}= {b_j}^{-1} {\theta_{ij}'}b_i ,
$$
where $\theta_{ij}={\alpha_j}^{-1} {\alpha_i} $ and $\theta_{ij}'={\alpha'_j}^{-1} {\alpha'_i}$ are the automorphisms
 of  $\Oh_X(U_{ij})\otimes A$ corresponding to the deformations $\mathcal{X}_A$ and  $\mathcal{X}_A'$, respectively. 

As regards the sheaves,  $\phi_{U_i}: {\sF_A} _{|U_i} \to f^*  {\sF_A'}_{|U_i}$ is an isomorphism of $\Oh_{\mathcal{X}_A}(U_i)$-modules, for any $i$. The structure  of $\Oh_{\mathcal{X}_A}$-modules on ${\sF_A} _{|U_i}$ is given by $\alpha_i$, while the structure on $f^*  {\sF_A'}_{|U_i}$ si given by $s\cdot x={\alpha'_i}^{-1} f^{-1} (s) x$, for every $s \in  \Oh_{\mathcal{X}_A}(U_i)$.
Therefore, $\phi_{U_i}$ has to satisfy  $\phi_{U_i} (sx)= s \phi_{U_i} (x)$, 
for all $s \in \Oh_{\mathcal{X}_A}(U_i)$, where
\[
\phi_{U_i}(sx)= \phi_{U_i} ( {\alpha_i} ^{-1}(s) x) \qquad  \mbox{ and }\qquad s \phi_{U_i}(x)=
{\alpha'_i}^{-1} f^{-1} (s)   \phi_{U_i}(x).
\]
Since ${\alpha_i}$ is an isomorphism, there exists $t \in \Oh_{X}(U_{i})\otimes A$ such that $ s=  {\alpha_i}  (t)$. Therefore, we have  
\[
\phi_{U_i}(tx)={\alpha'_i}^{-1}( f^{-1} ({\alpha_i}(t))  ) \phi_{U_i}(x)=
b_i(t) \phi_{U_i}(x),
\]
for any $t \in \Oh_{X}(U_{i})\otimes A$ and $x \in \sF_{|U_{ij}} \otimes A$.
Moreover, let $\psi_{ij} \in \Hom_\K(\sF_{|U_{ij}} \otimes A, \sF_{|U_{ij}} \otimes A)$ and $\psi_{ij}' \in \Hom_\K(\sF_{|U_{ij}} \otimes A, \sF_{|U_{ij}} \otimes A)$ the automorphism of  $\Oh_{\mathcal{X}_A}(U_{ij}) $-modules corresponding to the deformations $\sF_A$  and $  \sF_A'$, respectively;  then, the following diagram has to commute:
\begin{center}
$\xymatrix{ {\sF_{|U_i} \otimes A}_{|U_{ij}} \ar[r]^{{\phi_i}_{|U_{ij}}}\ \ar[d]^{\psi_{ij}} &{\sF_{|U_i} \otimes A}_{|U_{ij}}  \ar[d]^{\psi'_{ij}}  \\
            {\sF_{|U_j} \otimes A}_{|U_{ij}} \ar[r]^{{\phi_j}_{|U_{ij}}}  &    {\sF_{|U_j} \otimes A}_{|U_{ij}},\\ }$
\end{center}
i.e., $\psi'_{ij}{\phi_i}_{|U_{ij}}={\phi_j}_{|U_{ij}}\psi_{ij}$ that is 
 ${\phi_j}_{|U_{ij}}^{-1}\psi'_{ij}{\phi_i}_{|U_{ij}}=\psi_{ij}$

In conclusion, let  $(\mathcal{X}_A, \sF_A)$  and $(\mathcal{X}_A', \sF_A')$ be two locally trivial isomorphic deformations of  the pair $(X, \sF)$ over $A$, corresponding to the elements $\{(\theta_{ij},\psi_{ij} )\} $ and $\{(\theta'_{ij},\psi'_{ij} )\} $. The isomorphism given by   $f: \mathcal{X}_A \to \mathcal{X}_A'$  and  $ \phi: \sF_A \to f^*  \sF_A'$  corresponds to the elements 
$(b_i,\phi_{i}) \in \sD_A(\Oh_X(U_i)\otimes A,\sF(U_i)\otimes A)$, for any $i$, such that $
{b_i}^{-1} {\theta_{ij}'} b_j= \theta_{ij}$ and 
${\phi_i}^{-1} {\psi_{ij}'}\phi_j= \psi_{ij}$, for any $i$ and $j$.

Since we are in characteristic zero, according to Lemma \ref{lemma exp auto},  we  can take the logarithms  and write $(b_i,\phi_{i}) =e^{a_{i}}$,  where
$a_{i}\in\sD_\K(X,\sF)(U_{i})\otimes\mathfrak{m}_A$. The
condition of gluing  is equivalent to
\[e^{-a_i}e^{x_{ij}}e^{a_j}=e^{x'_{ij}}\;,\] 
where $\{e^{x_{ij}}=(\theta_{ij},\psi_{ij} )\} $ and $\{e^{x'_{ij}}=(\theta'_{ij},\psi'_{ij} )\} $.

\begin{theorem}\label{thereom D1 controlla def coppia loc trivial }
Let $ \sF$ be a coherent sheaf on  a scheme $X$ of finite type over  the field $\K$. Then, the
coherent sheaf   of  Lie algebras $ \sD_\K(X,\sF)$ controls the problem of locally trivial infinitesimal deformations of the pair $(X,\sF)$.
  
 
\end{theorem}

\begin{proof}

According to Definition \ref{definition fascio L controlla}, we have to prove that for any affine open cover $\mathcal{U}=\{U_i\}$ of $X$, there exists an isomorphism of functors of Artin rings
\[
\Def_{\Tot(\mathcal{U}, \sD_\K(X,\sF))} \to \Def^{\text{ lt}}_{(X, \sF)}.
\]
This follows from the previous consideration and from the explicit description of  the functor $\Def_{\Tot(\sU,\sD_\K(X,\sF))}$ given in Remark \ref{remark funtore lie algebre esplicito}:   for every  $A \in \Art_\K$, we have
\[
\Def_{\Tot(\mathcal{U}, \sD_\K(X,\sF))}(A) =\frac{\{ \{x_{ij}\} \in \prod_{i,j} \sD_\K(X,\sF)(U_{ij})
 \otimes
\mathfrak{m}_A   \ |\  e^{x_{jk}} e^{-x_{ik}} e^{x_{ij}}=1
\}}{\sim},
\]
where $x \sim y$ if and only if there exists
$\{a_i\} \in \prod_i\sD_\K(X,\sF)(U_{i })\otimes\mathfrak{m}_A$, such that $e^{{-a_i}_{|U_{ij}}}e^{x_{ij}}e^{{a_j}_{|U_{ij}}}=e^{y_{ij}}$.
\end{proof}

\begin{example}

Let $Z$ be a closed subscheme of an algebraic scheme $X$ and denote by $\Theta_X(-\log Z)\subset \Theta_X$
be the subsheaf of tangent vector fields that are tangent to $Z$. It is known that the sheaf 
of Lie algebras $\Theta_X(-\log Z)$  controls the locally trivial infinitesimal deformations of the pair $(X,Z)$
\cite{donacoppie}. By the same argument used at the beginning of Section \ref{sec.lifting}, we have  two morphisms of sheaves of Lie algebras
\[
\Theta_X(-\log Z) \xrightarrow{\;\Delta\;} \sD_\K(X,\Oh_Z) \xrightarrow{\;\alpha\;} \Theta_X(-\log Z),\qquad \alpha\circ \Delta=\Id.
\]
Indeed, the image of the anchor map $\alpha$ preserves the annihilator of the  sheaf $\Oh_Z$ and conversely every derivation $h$ preserving the ideal sheaf of $Z$ induces a derivation of pair $(h,h)$.

The geometrical interpretation of these morphisms is clear:
the morphism $\Delta$ controls the  natural transformation  that  associates with a locally trivial deformation of $(X,Z)$ the corresponding locally trivial deformation of $(X, \Oh_Z)$.
The anchor map $\alpha$ associates with a locally trivial deformation of the pair  $(X, \Oh_Z)$  
the corresponding locally trivial deformation of $(X, \Supp \Oh_Z)$.
 \end{example}\

\subsection{Infinitesimal deformations}

Next, we analyse all the infinitesimal deformations of a pair $(X, \sF)$. We assume that $X$ is smooth and projective, so that every infinitesimal deformation of $X$ is   locally trivial and  every coherent sheaf on $X$ has a finite locally free resolution.

Let $\sE^*\to \sF$ be a  finite locally free resolution of $\sF$, i.e., an exact sequence
\[
0 \to \sE^{-m} \stackrel{d}{\to} \sE^{-m+1}\stackrel{d}{\to}  \cdots \stackrel{d}{\to}\sE^{-1}  \stackrel{d}{\to}  \sE^{0} \to \sF \to 0,
\]
where   $\sE^j$ is a locally free sheaf, for any $j$. According to Example
\ref{example algebroid complex of sheaf}, we can consider the sheaf of DG-Lie algebras $\sD^*_\K(X,\sE^*)$ associated with the complex $\sE^*$.

\begin{lemma}\label{lemma D1 di risoluzioni quasi iso}
Let $\sF$ be a coherent sheaf on a smooth projective variety $X$ and
$f\colon\sE^*\to \sF$  and $g\colon\sG^*\to \sF$ two  finite locally free resolutions of $\sF$. Then, for any   open affine cover $\mathcal{U}=\{U_i\}$  of $X$,      the   DG-Lie algebras $\Tot(\mathcal{U}, \sD^*_\K(X,\sE^*))$ and $\Tot(\mathcal{U},\sD^*_\K(X,\sG^*))$ are quasi-isomorphic.

\end{lemma}

\begin{proof}

Since the variety is smooth and projective, by Hilbert Syzygy Theorem, there exist two bounded complexes  $\sH ^*$ and $ \sN ^* $ of locally free sheaves and a commutative diagram of complexes
\begin{center}
$\xymatrix{0 \ar[r]&\sE^* \oplus \sG^* \ar[r] \ar[dr]_{f+g} & \sH ^* \ar[d]^{h}  \ar[r] & \sN ^*   \ar[r] &0  \\
          &  &   \sF &\\ }$
\end{center}
such that $h: \sH ^* \to \sF$ is a quasi-isomorphism; in particular, the induced morphisms $\sE^*\to \sH^*$ and $\sG^*\to \sH^*$ are quasi-isomorphisms.

Let $U$ be an affine open subset of $X$.
Since the sections of a locally free sheaf of $\Oh_X$-module over $U$ is a projective $\Oh_X(U)$-module,
 the injection  $ \sE^*(U) \oplus \sG^*(U) \to \sH^*(U)$ is a cofibration, since  $\sH^*,\sN^*$  are a bounded complexes and therefore $\sH^*(U),\sN^*(U)$ are cofibrant complexes of $\Oh_X(U)$-modules. Since $U$ is affine 
the maps $ \sE^*(U) \oplus \sG^*(U) \to \sF(U)$ and $ \sH^*(U) \to \sF(U)$ are quasi-isomorphisms and then 
the induced maps  $\sE^*(U)\to \sH^*(U)$ and $\sG^*(U)\to \sH^*(U)$ are  trivial cofibrations.

Then, we have the following diagram of morphisms of coherent sheaves of DG-Lie algebras
\[
\sD^*_\K(X,\sE^*) \xleftarrow{\;}\sD^*_\K(X,\sE^*\xrightarrow{f} \sH^*)\xrightarrow{\;}\sD^*_\K(X,\sH^*)\xleftarrow{\;}\sD^*_\K(X,\sG^*\xrightarrow{g} \sH^*)\xrightarrow{\;}\sD^*_\K(X,\sG^*)\,.
\]
According to Propositions~\ref{propo D sono fasci coerenti} and  \ref{prop.trivialcofib}, for every open $U_i\in \sU$,
we have a diagram of quasi-isomorphisms of DG-Lie algebras
\[\begin{matrix}\xymatrix{  \sD^*_\K(X,\sE^*\xrightarrow{f} \sH^*)(U_i)\ar[dr]\ar[d] &   & \sD^*_\K(X,\sG^*\xrightarrow{g} \sH^*)(U_i) \ar[dl]\ar[d] &    \\
    \sD^*_\K(X,\sE^*)(U_i)        & \sD^*_\K(X,\sH^*)(U_i) & \sD^*_\K(X,\sG^*)(U_i)}
    \end{matrix}.\]
Finally,  applying the $\Tot$ functor, we conclude that $\Tot(\mathcal{U}, \sD^*_\K(X,\sE^*))$ and $\Tot(\mathcal{U},\sD^*_\K(X,\sG^*))$ are quasi-isomorphic
DG-Lie algebras.
\end{proof}

\begin{definition}\label{definizione T i}
Let $\sF$ be a coherent sheaf on a smooth projective variety $X$. For any $i\in \Z$, we define the coherent sheaves $\sT^i_{(X,\sF)}$ as the cohomology sheaves of $\sD^*_\K(X,\sE^*)$, for any  finite locally free resolution $\sE^*\to \sF$ of $\sF$:
\[
\sT^i_{(X,\sF)}:= \sH^i( \sD^*_\K(X,\sE^*))\, .\]
Similarly, we define the hyper-cohomology groups $T^i_{(X,\sF)}
= \H^i( X,\sD^*_\K(X,\sE^*))$.
Lemma \ref{lemma D1 di risoluzioni quasi iso} implies that the sheaves $\sT^i_{(X,\sF)}$ and the  groups $T^i_{(X,\sF)}$ are well defined, since  they do not depend on the choice of the resolution. \end{definition}

The short exact sequence \eqref{eq atyah sequence for complex} of complexes of coherent sheaves
\[ 0\to  \HOM_{\Oh_X}^*(\sE^*,\sE^*)\to \sD^*_\K(X,\sE^*)\to \Theta_X\to 0\]
gives a hypercohomology long exact sequence
\begin{equation}\label{equ.lungaipercoomologia}
0\to \Ext^0_X(\sF,\sF)\to T^0_{(X,\sF)}\to H^0(X,\Theta_X)\to \Ext^0_X(\sF,\sF)\to T^1_{(X,\sF)}\to\cdots \;.
\end{equation}

By the hypercohomology spectral sequence, we have
$E_2^{p,q}= H^p(X, \sT^q_{(X,\sF)}) \Rightarrow T^{p+q}_{(X,\sF)}$.
According to Example \ref{example successione esatta comologia per moduli}, we have that $ \sT^i_{(X,\sF)} =\EXT_{\Oh_X}^i(\sF,\sF)$, for $i\neq 0,1$, and  we have  an exact sequence of sheaves
\[ 0\to \HOM_{\Oh_X}(\sF,\sF)\to  \sT^0_{(X,\sF)}\to \Theta_X \to 
{\operatorname{\mathcal E}\!\!xt}_{\Oh_X}^1(\sF,\sF)\to \sT^1_{(X,\sF)}\to 0\,.\] 
In particular, the sheaf $\sT^i_{(X,\sF)}$ vanishes in the locus where $\sF$ has projective dimension
$<i$.

\medskip

{\bf Local Case.} First of all, we analyse the infinitesimal deformations of the pair $(X, \sF)$ in the local case.
Let $X = \Spec R$, where $R$ is smooth and affine over $\K$  and  $\sF =\widetilde{M}$, where $M$ is a finitely generated $R$-module.
Let $(E^*,d)$ be a finite complex of projective  $R$-module, such that the sequence
\[
0 \to  E^{-m} \stackrel{d}{\to}  E^{-m+1}\stackrel{d}{\to}  \cdots \stackrel{d}{\to}E^{-1}  \stackrel{d}{\to}  E^{0} \to M  \to 0,
\]
is exact. Finally, consider  the DG-Lie algebra  $D^*_\K(R,E^*)$.

\begin{proposition}\label{proposition DGLA local}
In the notation above, the DG-Lie algebra $D^*_\K(R,E^*)$ controls the infinitesimal deformations of the pair $(X, \sF )$,
i.e., there exists an isomorphism  of deformation functors
\[
\Def_{D^*_\K(R,E^*) } \cong
\Def_{(X, \sF )}.
\]
\end{proposition}

\begin{proof} 
The  variety $X= \Spec R$ is smooth and affine and so it has only trivial infinitesimal deformations \cite[Theorem 1.2.4]{Sernesi}. Therefore,
for any $A\in \Art_\K$,
 any    infinitesimal deformation of the pair $(X,\sF=\widetilde{M})$ over $A$ is of the form $(X \times \Spec A,F_A=\widetilde{M}_A)$, where $M_A$ is an $R \otimes A$-module,   that is  $A$-flat, together with a morphism $\pi: M_A \to M$, inducing an isomorphism $M_A\otimes_A \K \to M$ \cite[Section 1]{FIM}. 

The flatness condition allows to lift the relations; therefore any deformation $M_A$ of $M$ over $A$ corresponds to an exact sequence
$E^*\otimes A \to  M_A$ that reduces to $E^* \to M$, when tensored with $\K$, i.e., any deformation $M_A$ of $M$ over $A$ corresponds to a deformed complex $(E^*\otimes A, d_A)$ and any of these complexes is of the form $(E^*\otimes A, d+x)$, for  $x \in \MC_{D^*_\K(R,E^*)}(A)=\MC_{\Hom^*_\K(E^*,E^*)}(A)$.

Two deformations  $(X \times \Spec A,\sF_A=\widetilde{M}_A)$ and  $(X \times \Spec A,\sF'_A=\widetilde{M}'_A)$ of $(X,\widetilde{M})$ over $A$   are isomorphic if there exist an isomorphism  of
deformations $f\colon X \times \Spec A \to X \times \Spec A$  and an isomorphism $  \sF_A \to f^*  \sF_A'$ of $R \otimes A$-modules.

  In particular, the isomorphism of the modules lifts to an isomorphism of the deformed complexes and vice versa.
Then, an isomorphism is given by a pair $(\theta, \psi)$ where 
$\theta: R\otimes A \to R\otimes A$ is an isomorphism that lifts the identity of $R$ and $\psi \colon (E^*\otimes A, d+x) \to (E^*\otimes A, d+x')$ is an isomorphism, lifting the identity,  such that $\psi(rm)= \theta(r)\psi(m)$, for any $r \in R\otimes A$ and any $m \in E^*\otimes A$.

As in Lemma \ref{lemma exp auto}, for any $A\in \Art_\K$ the group $ \exp(D^0_\K(R,E^*)\otimes\mathfrak{m}_A)$ is naturally isomorphic to the group of $A$-automorphisms of the pair $(R\otimes A, E^* \otimes A)$ lifting the identity on $(R,E^*)$; 
therefore, it  corresponds to the group of the isomorphisms of deformations.

\end{proof}

\begin{example}[cf. {\cite[Proposition A3]{FGvS}}]
Assume that $X$ is  a smooth surface, $D\subset X$ a reduced divisor and
$\sF$ a sheaf of  $\Oh_D$-module.
Assume that for every point $p \in D$, the stalk $\sF_p$ is a torsion free 
$\Oh_{D,p}$-module.
Then, by Auslander-Buchsbaum Theorem \cite[Theorem 19.1]{Matsubook},
the  projective dimension of every stalk of the sheaf $\sF$ is  at most 1.
In particular,  $\sT^i_{(X,\sF)} =\EXT_{\Oh_X}^i(\sF,\sF)=0$, for every $i\geq 2$. 

If $X$ is affine, we have $\Ext_X^2(\sF,\sF)= H^0 ( \EXT_{\Oh_X}^2(\sF,\sF))=0$,  and so  the pair $(X,\sF)$ has unobstructed deformations.

\end{example}

{\bf Global Case.}  Finally, we analyse  all the infinitesimal deformations of a   pair $(X,\sF)$.  Essentially, we
 have to glue together the computations about the locally trivial infinitesimal deformations  (Theorem \ref{thereom D1 controlla def coppia loc trivial }) and the  deformations in the local case (Proposition \ref{proposition DGLA local}).

\begin{theorem}\label{thereom D1 controlla def coppia X F}
Let $ \sF$ be a coherent sheaf on  a projective smooth variety $X$ and $\sE^*\to \sF$  a  finite locally free resolution of $\sF$.
Then, the
coherent sheaf   of DG-Lie algebras $\sD^*_\K(X,\sE^*)$ controls the problem of  deformations of the pair $(X,\sF)$.
In particular, $T^1_{(X,\sF)}$ is the tangent space and $T^2_{(X,\sF)}$ is the obstructions  space for $\Def_{(X, \sF)}$.

\end{theorem}

\begin{proof} The proof follows the general lines already used in \cite{FIM}.
According to Definition \ref{definition fascio L controlla}, we have to prove that for any affine open cover $\mathcal{U}=\{U_i\}$ of $X$, there exists an isomorphism of functors of Artin rings
\[
\Def_{\Tot(\mathcal{U},\sD^*_\K(X,\sE^*) )} \to \Def_{(X, \sF)}.
\]

According to the isomorphism given in Equation \eqref{eq iso funtori Def H1}, it is enough to show
 that there exists an isomorphism $H^1_{\rm sc}(\exp
(\sD^*_\K(X,\sE^*) (\mathcal{U}))) \to \Def_{(X, \sF)}$. Therefore,  for any $A \in \Art_\K$, 
we have to prove that every element
in  $Z^1_{\rm sc} (\exp (\sD^*_\K(X,\sE^*) (\mathcal{U})))(A)$ corresponds to an infinitesimal deformation of the 
pair $(X, \sF)$ over $A$ and that two elements are  equivalent if and only if the corresponding deformations  are isomorphic.
This follows from the local study analysed in Proposition \ref{proposition DGLA local} and a gluing procedure as in  Theorem \ref{thereom D1 controlla def coppia loc trivial }. Indeed, an element 
$(l,m) \in Z^1_{\rm sc} (\exp (\sD^*_\K(X,\sE^*) (\mathcal{U}))) (A)$ gives a deformation of the pair $(X,\sF)$ as 
gluing of deformations on each $U_i$. We only stress the fact that the  gluing condition on the isomorphisms involves
an element in $\prod_{i,j,k}  \HOM^{-1}_{\Oh_X}( \sE^*,\sE^*)  (U_{ijk}) \otimes \mathfrak m_A $; this is due to the 
fact that we do not have to glue the deformed complexes but rather their cohomology to get a sheaf (see \cite[Section 2]{FIM} for more details about this).  
As regards the equivalence relation $\sim$, the first condition is the isomorphism of the induced deformation
on each open, the second condition gives the gluing of the local isomorphism to have a global isomorphism of the
induced deformations.

More explicitly, an element $(l,m) \in Z^1_{\rm sc} (\exp (\sD^*_\K(X,\sE^*) (\mathcal{U})))(A)$ is given by the 
sequences $l=\{l_i\} \in \prod _i \HOM^1_{\Oh_X}( \sE^*,\sE^*) (U_i) \otimes \mathfrak m_A  
$ and $m = \{(h_{ij}, u_{ij})\} \in \prod_{i,j}  \sD^0_\K(X,\sE^*)(U_{ij})\otimes \mathfrak m_A $ such that:

\begin{enumerate}
\item for any $i$, the element $l_i$ satisfies the Maurer-Cartan equation, i.e., $dl_i+ \frac{1}{2}[l_i,l_i]=0$;
\item for any $i$ and $j$, the restriction of $l_i$ and $l_j$ to $U_{ij}$ are gauge equivalent under $m_{ij}$, i.e.,  $l_i|_{U_{ij}}=e^{m_{ij}}*l_j|_{U_{ij}}$.
\item for any $i,j$ and $k$, we have the following conditions:
\[
h_{jk} |_{U_{ijk}} \bullet - h_{ik}|_{U_{ijk}} \bullet h_{ij}|_{U_{ijk}} = 0
\]
\[
u_{jk} |_{U_{ijk}} \bullet - u_{ik}|_{U_{ijk}} \bullet u_{ij}|_{U_{ijk}} =
[d + l_j |_{U_{ijk}} , n_{ijk}],
\]
for some $n=\{n_{ijk}\} \in \prod_{i,j,k}  \HOM^{-1}_{\Oh_X}( \sE^*,\sE^*)  (U_{ijk}) \otimes \mathfrak m_A $.
\end{enumerate}
According to the local study, condition (1) and (2) imply that every $l_i$ defines a deformation of the pair $(X, \sF)$ on $U_i$ and that the deformations induced by $l_i$ and $l_j$  on $U_{ij}$ are isomorphic.
Finally, the third condition implies that  we can glue together the local deformations to get a global deformation the pair  $(X, \sF)$ over $A$.
Indeed, the former condition is the gluing of the locally trivial deformation of $U_i$ over the double intersections to have a   deformation of $X$ over $A$; the latter regards the gluing of the local deformations of the sheaf $\sF$.
We only stress the fact that the isomorphisms that glues the restriction of the deformed complexes
 $(\sE^*_{U_i} \otimes A, d+l_i)$ and $(\sE^*_{U_j} \otimes A, d+l_j)$, for any $i$ and $j$ satisfy the cocycle condition only up to homotopy. Indeed, we do not have to glue the restriction of the deformed complexes of sheaves together but rather their cohomology. Therefore, the gluing isomorphism have to satisfy the cocycle condition only up to homotopy.

As regard the equivalence relation, suppose that 
$(l_0=\{l_{0,i}\},m_0=   \{(h_{0,ij}, u_{0,ij})\})$ and $(l_1=\{l_{1,i}\},m_1=   \{(h_{1,ij}, u_{1,ij})\})  \in Z^1_{\rm sc}  (\exp
(\sD^*_\K(X,\sE^*) (\mathcal{U})))(A)$
are equivalent under the relation  $\sim$, then we need to prove that the induced deformations are isomorphic. 
The equivalence implies the  existence of $a =\{g_{i}, v_i\} \in  \prod_{i}  \sD^0_{\K}(X,\sE^*)(U_{i})\otimes {\mathfrak m}_A$ and $b= \{b_{ij}\}  \in \prod_{i,j}  \HOM^{-1}_{\Oh_X}( \sE^*,\sE^*)  (U_{ij}) \otimes \mathfrak m_A $ such that
\[
\begin{cases}
e^a * l_0=l_1\\
- m_0\bullet -{a_i}_{|U_{ij}} \bullet m_1
\bullet {a_j}_{|U_{ij}} =db+[{l_{0,j}}_{|U_{ij}} ,b].
\end{cases}
\] 
The first condition implies that the deformations of the pair induced by $l_0$ and $l_1$ are isomorphic  locally on $U_i$;
the second condition provides the gluing of the isomorphisms along double intersections. (We have a strict cocycle condition for the isomorphisms  of the deformations of $X$ and, as before, we have a homotopy 
cocycle condition for the gluing of the deformed complexes.)

As regard the tangent and obstruction space for $\Def_{(X, \sF)}$, we have that $H^1(\Tot(\mathcal{U},\sD^*_\K(X,\sE^*)))$ is the tangent space and $H^2(\Tot(\mathcal{U},\sD^*_\K(X,\sE^*)))$ is an obstruction  space.
According to Definition \ref{definizione T i}  and to Equality \eqref{eq. comologia tot is hypercoho}, we have that 
\[
H^i(\Tot(\sD^*_\K(X,\sE^*) (\mathcal{U}))) \cong \H^i(X, \sD^*_\K(X,\sE^*))=T^i_{(X,\sF)}.
\]

\end{proof}

\begin{remark} In the situation of Theorem~\ref{thereom D1 controlla def coppia X F}, consider the 
sheaf of truncated DG-Lie algebras 
\[ \sigma_{\le0}\sD_{\K}^*(X,\sE^*)=\sD_{\K}(X,\sE^*)\oplus_{i<0}\HOM^i_{\Oh_X}(\sE^*,\sE^*)\subset
\sD_{\K}^*(X,\sE^*)\,.\]
We have seen in Example~\ref{example successione esatta comologia per moduli}, that the natural map 
$\sD_{\K}(X,\sE^*)\to \sD_{\K}(X,\sF)$ induces a  surjective quasi-isomorphism of coherent sheaves  of DG-Lie algebras 
$\sigma_{\le0}\sD_{\K}^*(X,\sE^*)\to \sD_{\K}(X,\sF)$. 
Therefore, also the sheaf of truncated DG-Lie algebras $\sigma_{\le0}\sD_{\K}^*(X,\sE^*)$ controls 
the locally trivial deformations of the pair, and the obvious natural transformation 
$\Def_{(X,\sF)}^{\text{ lt}}\to \Def_{(X,\sF)}$ is induced by the inclusion 
$\sigma_{\le0}\sD_{\K}^*(X,\sE^*)\subset \sD_{\K}^*(X,\sE^*)$. 
\end{remark}

We are now ready to prove the last item of Theorem~\ref{thm.maintheorem}.
  
\begin{proposition}\label{prop.prorepcrit}
Let $ \sF$ be a coherent sheaf on  a projective smooth variety $X$. Assume that $H^0(X,\Theta_X)=0$ and 
$\Ext_X^0(\sF,\sF)=\K$, then the functor $\Def_{(X,\sF)}$ is pro-representable.
\end{proposition}

\begin{proof} Let $\sE^*\to \sF$  be a  finite locally free resolution of $\sF$. According to 
\eqref{equ.lungaipercoomologia} the vector space $T^0_{(X,\sF)}$ is one-dimensional generated by the identity on the complex
$\sE^*$. Therefore, for every open affine cover 
$\sU$ the vector space 
$H^0(\Tot(\mathcal{U},\sD^*_\K(X,\sE^*)))$ is one-dimensional and generated by the identity on 
$\sE^*_{|U_{i_0\cdots i_k}}$ for every $k\ge 0$ and every $i_0,\ldots,i_k$.
In particular the DG-Lie algebra $\Tot(\mathcal{U},\sD^*_\K(X,\sE^*))$ satisfies the condition of Lemma~\ref{lem.criteriouniversal} and the conclusion follows by Theorem~\ref{thereom D1 controlla def coppia X F}.
\end{proof}

Let  $X$ be a projective smooth variety, $ \sF$  a coherent sheaf on $X$ and
$\sE^* \to \sF$ a finite locally free resolution. By definition  $ \det \sF =\det \sE^*$, and we have defined a morphism of sheaves of DG-Lie algebras (see Equation \eqref{equ.traccia alldegree}):
\[ 
\Tr\colon \sD_\K^*(X,\sE^*)\to \sD_\K(X,\det\sF)\,.
\]
By Theorem \ref{thereom D1 controlla def coppia X F}, the sheaf $\sD_\K^*(X,\sE^*)$ controls the deformations of the pair $(X, \sF)$, while the sheaf $\sD_\K(X,\det \sF)$ controls the deformations of the pair $(X,  \det\sF)$.

We can interpret this morphism as a natural transformation of deformation functors.
Some deformations of the pair $(X,\sF)$ corresponds to a deformation  of $X$ and of the complex  
$\sE^*$; then with these deformations we can easily associate a deformation of the pair $(X,\det \sF=\det \sE^*) $, considering
the determinant of the deformed complex. 
In general, not all the deformations of $\sF$ corresponds to a deformations of the complex, but this is true only locally on $X$
and the deformations of the complex glues only in cohomology. Anyway, we can still define a deformation of the pair 
$(X,\det\sF)$, as we view in the following lemma.

 \begin{lemma}
Let  $X$ be a projective smooth variety and  $ \sF$  a coherent sheaf on $X$. Then, there exists a commutative diagram
of deformation functors
\begin{center}
$\xymatrix{ \Def_{(X, \sF )}\ar[rr]^{\Tr} \ar[dr] & &\Def_{(X, \det \sF )} \ar[ld]   \\
            &    \Def_{X }&\\ }$
\end{center}
where the diagonal arrows are the forgetful natural transformations, 
and the natural transformation $\Tr$ is induced by the
trace map  $\Tr\colon \sD^*_{\K}(X,\sE^*)  \to \sD^*_{\K}(X, \det \sF)$, for any finite locally free resolution $\sE^* \to \sF$.

\end{lemma}
 
  \begin{proof}
Let   $\mathcal{U}=\{U_i\}$ be an open affine cover  for $X$, such that every sheaf  $\sE^k$ is free on every $U_i$. 
Fix $A \in \Art_\K$ and  consider a deformation $(X_A, \sF_A) \in  \Def_{(X, \sF )}(A)  $.    According to Theorem~\ref{thereom D1 controlla def coppia X F},
this deformation corresponds to an element   $(l,m) \in Z^1_{\rm sc} (\exp (\sD^*_\K(X,\sE^*) (\mathcal{U})))(A)$, where  
  $l=\{l_i\} \in \prod _i \HOM^1_{\Oh_X}( \sE^*,\sE^*) (U_i) \otimes \mathfrak m_A  
$ and $m = \{(h_{ij}, u_{ij})\}  \in \prod_{i,j}  \sD^0_\K(X,\sE^*)(U_{ij})\otimes \mathfrak m_A $. In particular, 
for any $i,j$ and $k$, we have the following conditions:
\[
u_{jk} |_{U_{ijk}} \bullet - u_{ik}|_{U_{ijk}} \bullet u_{ij}|_{U_{ijk}} =
[d + l_j |_{U_{ijk}} , n_{ijk}],
\]
for some $n=\{n_{ijk}\}  \in \prod_{i,j,k}  \HOM^{-1}_{\Oh_X}( \sE^*,\sE^*)  (U_{ijk}) \otimes \mathfrak m_A $.
Applying the trace morphisms to $m$, for any $i$ and $j$, we have
\[
\Tr(h_{ij}, u_{ij})=(h_{ij}, \Tr(u_{ij})) \in \sD^0_\K(X,\det\sF)(U_{ij})\otimes \mathfrak m_A 
\]
such that
\[
\Tr(u_{jk} |_{U_{ijk}}) \bullet - \Tr(u_{ik}|_{U_{ijk}}) \bullet 
\Tr(u_{ij}|_{U_{ijk}}) =\Tr([d + l_j |_{U_{ijk}} , n_{ijk}])=0.
\]
This implies that the element $\Tr(m)=\{(h_{ij}, \Tr(u_{ij})\} \in \prod_{i,j} \sD^0_\K(X,\det\sF)(U_{ij})\otimes \mathfrak m_A $
satisfies the cocycle condition on triple intersections and so, by Theorem \ref{thereom D1 controlla def coppia loc trivial }, it 
 defines a deformation of the pair $(X, \det \sF)$ over $A$.

An analogous computation shows that the isomorphism class of the induced deformation of $(X, \det \sF)$ does not depend on the isomorphism class of the deformation of $(X, \sF)$.
Indeed, let $(X_A, \sF_A)$ and $(X'_A, \sF'_A)$ be isomorphic deformations, corresponding to the elements 
$(l_0=\{l_{0,i}\} ,m_0=   \{(h_{0,ij}, u_{0,ij})\} )$ and $(l_1=\{l_{1,i}\} ,m_1=   \{(h_{1,ij}, u_{1,ij})\} )  \in Z^1_{\rm sc}  (\exp
(\sD^*_\K(X,\sE^*) (\mathcal{U})))(A)$. According to Theorem \ref{thereom D1 controlla def coppia X F}, the isomorphism corresponds
to the existence of  
\[a =\{g_{i}, v_i\} \in  \prod_{i}  \sD^0_{\K}(X,\sE^*)(U_{i})\otimes {\mathfrak m}_A\quad\text{and}\quad b= \{b_{ij}\} \in \prod_{i,j}  \HOM^{-1}_{\Oh_X}( \sE^*,\sE^*)  (U_{ij}) \otimes \mathfrak m_A\] 
such that
\[
\begin{cases}
e^a * l_0=l_1\\
- m_0\bullet -{a_i}_{|U_{ij}} \bullet m_1
\bullet {a_j}_{|U_{ij}} =db+[{l_{0,j}}_{|U_{ij}} ,b].
\end{cases}
\] 
In particular, applying the trace 
morphisms  for any $i$ and $j$, we have
\[
-\Tr (m_0)\bullet -\Tr ({a_i}_{|U_{ij}}) \bullet \Tr( m_1)
\bullet \Tr ({a_j}_{|U_{ij}} )=0
\in \sD^0_\K(X,\det\sF)(U_{ij})\otimes \mathfrak m_A\,. 
\]
By  Theorem \ref{thereom D1 controlla def coppia loc trivial }, this implies that the induced deformations of  $(X, \det \sF)$ are isomorphic.
 \end{proof}
  
\begin{remark}
As a particular case, if we only consider  the deformations of a coherent sheaf $ \sF$ on a fixed projective variety
$X$, then trace induces  a morphism of functors  
\[ \Def_{\sF } \to \Def_{\det \sF }.
\]
In this case, the induced morphisms in cohomology are the ones already analysed in \cite[Theorem 3.23]{casson}.
\end{remark} 

\begin{theorem} \label{teo finale su traccia}
Let  $X$ be a projective smooth variety and  $\sE^*$ a locally free finite resolution of a coherent sheaf $\sF$.   Consider the linear maps $\Tr^i\colon  \Ext_X^i (\sF,\sF) \to H^i(X,\Oh_X)$ induced by the 
morphism of sheaves of DG-Lie algebras 
$\Tr\colon \HOM^*_{\Oh_X}(\sE^*,\sE^*)\ \to \Oh_X$.\par

If the map 
$\Tr^1$ is surjective and the map $\Tr^2$ is injective, then the induced natural transformation  $\Def_{(X,\sF)}\to \Def_{(X,\det \sF)}$ is smooth. If $\Tr^2$ is injective and $\Def_{(X,\det \sF)}$ is unobstructed, then $\Def_{(X,\sF)}$ is unobstructed.
\end{theorem}
 
By the exact  sequence \eqref{equ.ext0}, the above assumptions on $\Tr^1$ and $\Tr^2$ are equivalent to assuming  $\Ext_X^2(\sF,\sF)_0=0$.  
 
 \begin{proof}
By Theorem \ref{thm.tracediagram},   there exists a commutative diagram with exact rows and columns of morphisms of coherent sheaves of DG-Lie algebras
\[ \xymatrix{&0\ar[d]&0\ar[d]&&\\
&\HOM^*_{\Oh_X}(\sE^*,\sE^*)_0\ar@{=}[r]\ar[d]&\HOM^*_{\Oh_X}(\sE^*,\sE^*)_0\ar[d]&&\\
0\ar[r]&\HOM^*_{\Oh_X}(\sE^*,\sE^*)\ar[d]^{\Tr}\ar[r]&\sD^*_\K(X,\sE^*)\ar[d]^{\Tr}\ar[r]^-{\alpha}&\Theta_{X}\ar[r]\ar@{=}[d]&0\;\;\\
0\ar[r]&\Oh_X\ar[d]\ar[r]&\sD_\K(X,\det \sE^*)\ar[d]\ar[r]^-{\alpha}&\Theta_{X}\ar[r]&0\\
&0& \ \ 0\;.&&}\]
The hypothesis on the morphisms $\Tr\colon \HOM^*_{\Oh_X}(\sE^*,\sE^*)\ \to \Oh_X$ implies that the same holds for the morphisms induced by $\Tr\colon \sD_\K^*(X,\sE^*)\to \sD_\K(X,\det\sF)$.
Then, it is enough to apply the Standard Smoothness Criterion  \cite[Theorem 4.11]{ManettiSeattle} to the corresponding  morphism of functors $\Tr\colon\Def_{(X,\sF)}\to \Def_{(X,\det \sF)}$. 

If $\Def_{(X,\det \sF)}$ is unobstructed, then $\Tr^2\colon T^2_{(X,\sF)}\to T^2_{(X,\det \sF)}$ annihilates all the obstructions and its injectivity implies that also $\Def_{(X,\sF)}$ is unobstructed.
 \end{proof}

\begin{acknowledgement}
D.I. wish
to thank the Dipartimento di Matematica  \lq\lq G. Castelnuovo\rq\rq of
Sapienza Universit\`a di Roma for the hospitality during the preparation of this paper. 
D.I. acknowledges the  support of Fondi di Ateneo dell'Universit\`a di Bari; 
M.M. acknowledges the    support by Italian MIUR under PRIN project 2015ZWST2C \lq\lq Moduli spaces and Lie theory\rq\rq. 
We wish to thank the referees for useful comments and for suggestions improving the presentation of the paper.

 \end{acknowledgement}

\end{document}